\documentclass[11pt]{amsart}

\usepackage{amsthm, amsfonts, amssymb,pinlabel,graphicx}
\usepackage[usenames,dvipsnames]{xcolor}
\usepackage{pinlabel}
\usepackage[pdftex,%
  final,%
  colorlinks=true,%
  linkcolor=NavyBlue,%
  citecolor=NavyBlue,%
  filecolor=NavyBlue,%
  menucolor=NavyBlue,%
  urlcolor=NavyBlue,%
  bookmarks=true,%
  bookmarksdepth=3,%
  bookmarksnumbered=true,%
  bookmarksopen=true,%
  bookmarksopenlevel=2,%
]{hyperref}

\numberwithin{equation}{section}
\graphicspath{{Figures/}}

\theoremstyle{plain}
\newtheorem{thm}{Theorem}[section]
\newtheorem{cor}[thm]{Corollary}
\newtheorem{lem}[thm]{Lemma}

\newtheorem{prop}[thm]{Proposition}
\newtheorem{conj}[thm]{Conjecture}

\theoremstyle{definition}
\newtheorem{defn}[thm]{Definition}

\theoremstyle{remark}
\newtheorem{rem}[thm]{Remark}
\newtheorem{ex}[thm]{Example}

\newcommand{\rr}{\ensuremath{\mathbb{R}}}
\newcommand{\zz}{\ensuremath{\mathbb{Z}}}

\newcommand{\leg}{\ensuremath{\Lambda}}

\newcommand{\dfn}[1]{{\textbf {#1}}}
\DeclareMathOperator{\tb}{tb} 
\DeclareMathOperator{\lk}{lk}
\DeclareMathOperator{\rlk}{rlk}

\DeclareMathOperator{\ind}{ind}

\def\js#1{{\textcolor{Cerulean}{#1}}}

\begin{document}

\title{Non-Orientable Lagrangian Fillings of Legendrian Knots} 

\begin{abstract}
We investigate when a Legendrian knot in standard contact $\rr^3$ has a non-orientable exact Lagrangian filling. We prove analogs of several results in the orientable setting, develop new combinatorial obstructions to fillability, and determine when several families of knots have such fillings. In particular, we determine completely when an alternating knot (and more generally a plus-adequate knot) is decomposably non-orientably fillable, and classify the fillability of most torus and 3-strand pretzel knots. We also describe rigidity phenomena of decomposable non-orientable fillings, including finiteness of the possible normal Euler numbers of fillings, and the minimization of crosscap numbers of fillings, obtaining results which contrast in interesting ways with the smooth setting.
\end{abstract}

\date{\today}

\author[L. Chen]{Linyi Chen} 
\address{Google, Inc.} \email{chenlinyi1996@gmail.com}

\author[G. Crider-Phillips]{Grant Crider-Phillips} 
\address{University of Oregon} \email{criderg@uoregon.edu}

\author[B. Reinoso]{Braeden Reinoso} \address{Boston College} \email{reinosob@bc.edu}

\author[J. Sabloff]{Joshua M. Sabloff} \address{Haverford College} \email{jsabloff@haverford.edu} \thanks{JS, LC, and LY were
partially supported by NSF grant DMS-1406093 during the preparation of this paper; GCP and BR were partially supported by grants from Haverford College.}

\author[L. Yao]{Leyu Yao} 
\address{University of Cambridge} \email{ly339@cam.ac.uk}

\maketitle

\section{Introduction}
\label{sec:intro}


The motivating question for this paper is: given a Legendrian link $\leg \subset (S^3, \xi_0)$, how does the existence of an exact Lagrangian filling $L \subset (B^4,\omega_0)$, orientable or not, restrict the smooth knot type of $\leg$?  We say that a smooth knot $K$ is (orientably or non-orientably) \dfn{Lagrangian fillable} if it has  a Legendrian representative that has an exact (orientable or non-orientable) Lagrangian filling.

Being \emph{orientably} fillable is a strong condition on a smooth knot: the maximal Thurston-Bennequin number is realized by the negative Euler characteristic of a Lagrangian filling \cite{chantraine}, every Lagrangian filling minimizes the $4$-ball genus \cite{chantraine},  the HOMFLY bound on the maximal Thurston-Bennequin number is sharp, and the knot is quasipositive \cite{positivity} (using \cite{bo:qp, yasha:lagr-cyl}).  It is not yet clear exactly which smooth knots are orientably Lagrangian fillable, though the class lies strictly between the set of positive knots \cite{positivity} --- and even some almost positive knots \cite{tagami:filling-almost-pos} --- and the set of quasipositive knots.  Note that the motivating question introduces an interesting distinction between Legendrian and transverse knots:  the class of transverse knots with symplectic fillings coincides with quasipositive knots \cite{bo:qp}, though not all quasipositive knots are Lagrangian fillable.

Comparatively less is known about non-orientably fillable knots.  Atiponrat \cite{atiponrat:obstruction} developed an obstruction to the existence of a decomposable exact Lagrangian filling using the parity of the number of ``clasps'' of a normal ruling, and used that obstruction to prove that a maximal $tb$ representative of the $(4,-(2n+5))$ torus knot has no exact Lagrangian filling \cite{atiponrat:torus}.  Capovilla-Searle and Traynor \cite{cst:non-ori-cobordism} studied non-orientable, but not necessarily exact, Lagrangian endocobordisms; in developing obstructions to such endocobordisms, they produced examples of exact non-orientable fillings in a few families of knots, including some twist knots and $(p,-2)$ torus knots.

In this paper, we establish techniques for investigating non-orientable Lagrangian fillings and apply those techniques to families of knots. In parallel to the features of the orientable case, we begin by making connections between the classical invariants of a Legendrian knot and the \dfn{normal Euler numbers} of its non-orientable Lagrangian fillings:

\begin{prop}
\label{prop:tb-euler-filling}
If $L$ is a Lagrangian filling of $\leg$ with normal Euler number $e(L)$, then
	\[ tb(\leg) = - \chi(L) - e(L). \]
\end{prop}

Even though a knot has smooth fillings that realize infinitely many normal Euler numbers, we prove:

\begin{prop} \label{prop:finite-euler}
	For any given Legendrian knot, only finitely many Euler numbers may be realized by exact decomposable non-orientable Lagrangian fillings.
\end{prop}

That said, there is a sequence of knots for which the corresponding sequence of sets of Euler numbers realized by Lagrangian fillings grows without bound; see Theorem~\ref{thm:wh}. 

Throughout the paper, in parallel to the fact that orientable Lagrangian fillings minimize the smooth $4$-genus, we present evidence which suggests that any non-orientable exact Lagrangian filling with normal Euler number $e$ minimizes the crosscap number among smooth fillings with normal Euler number $e$. However, the proof of such a minimization result is made difficult by the lack of an adjunction inequality for non-orientable surfaces.

Just as there is a connection between the HOMFLY polynomial and orientable fillings, the existence of a non-orientable fillings implies the sharpness of the Kauffman polynomial bound on the maximal Thurston-Bennequin number; see Proposition~\ref{prop:ruling-obstr}. 

Finally, extending Atiponrat's work on rulings \cite{atiponrat:obstruction}, we develop techniques for obstructing decomposable non-orientable Lagrangian fillings, including the development of an easily computable obstruction we call the \dfn{resolution linking number}.  This obstruction is applied to prove Theorem~\ref{thm:torus}, below.

In the second half of the paper, we combine the obstructions discussed above with constructions in families to reveal hints of the geography of non-orientably fillable smooth knots. We start by proving:

\begin{thm}
\label{thm:positive}
	If $K$ is positive, then $K$ is orientably, but not non-orientably, decomposably fillable.
\end{thm}

We can completely characterize Lagrangian fillability of alternating knots. Cornwell, Ng, and Sivek \cite{cns:obstr-concordance} showed that an alternating knot is orientably fillable if and only if it is positive; we extend their analysis to non-orientable fillings of alternating knots.

\begin{thm}
\label{thm:alternating}
	If $K$ is alternating, then $K$ is non-orientably decomposably fillable if and only if $K$ is not positive.
\end{thm}

This theorem will follow from a more general result about plus-adequate knots; see Theorem~\ref{thm:plus-ad}.

From these results, one might begin to suspect that non-orientable fillability is complementary to some notion of positivity, but it turns out that the existence of non-orientable fillings is more subtle. First, there are Legendrian knots with both orientable and non-orientable fillings; see Example~\ref{ex:8-21-filling}.  Second, the following two families contain knots which realize every combination of quasipositivity and non-orientable fillability:

\begin{thm}
\label{thm:torus}
	For $p$ and $q$ relatively prime with $|p|>q$, let $T(p,q)$ denote the $(p,q)$ torus knot.
	\begin{enumerate}
	\item $T(p,q)$ is orientably fillable if and only if $p>q>0$.
	\item $T(p,2)$ is non-orientably fillable if $p<0$.
	\item $T(p,q)$ is not fillable if $p<0$ and $q$ is odd.
	\item $T(p,q)$ is not decomposably fillable if $p<0$ and $4|q$.
	\end{enumerate}
\end{thm}

\begin{thm}
\label{thm:pretzel}
Let $p_1,p_2,p_3>0$ and let $K\neq P(-p_1,-p_2,p_2-1)$ be a 3-stranded pretzel knot. Then $K$ is decomposably non-orientably fillable if and only if $K$ is isotopic to a pretzel knot of one of the following forms:
	\begin{enumerate}
	\item $P(p_1,p_2,p_3)$,
	\item $P(-p_1,-p_2,-p_3)$ with exactly one of the $p_i$ even
	\item $P(-p_1,p_2,p_3)$ with $p_1$ odd,
	\item $P(-p_1,-p_2,p_3)$ with $p_1 \geq p_2$, $p_1$ odd, and either $p_2=p_3=1$ or $p_2<p_3$, or
	\item $P(-p_1,-p_2,p_3)$ with $p_1 \geq p_2$, $p_2 \geq p_3+2$, and one of the $p_i$ even.
	\end{enumerate}
\end{thm}

\begin{rem}
The reason we must exclude $K=P(-p_1,-p_2,p_2-1)$ from Theorem \ref{thm:pretzel} is because we do not have a max-$tb$ front diagram for this family; whenever the front in Figure \ref{pretzel_fronts45} is $tb$-maximal, the conditions in case (4) of the theorem apply so that $K$ is not decomposably non-orientably fillable.
\end{rem}

Even with the evidence above, it is difficult to form a precise conjectural description of smooth knot types that are non-orientably fillable.  While one would na\"ively hope to parallel the conjecture that orientable fillings are related to the (quasi)positive hierarchy and the sharpness of the HOMFLY bound, such hopes do not survive encounters with examples.  While there are no positive knots with non-orientable fillings, there exist examples of non-orientably fillable knots that are neither quasipositive nor negative (e.g. $4_1$) and examples that are quasipositive (e.g. $m(8_{21})$, which has both orientable and non-orientable fillings).  In the other direction, the results above about negative torus knots and some pretzel knots show that not all negative knots are non-orientably fillable. In fact, it seems that non-orientably fillable knots are non-positive and, roughly, not \emph{too} negative. 

The remainder of the paper is organized as follows:  in Section~\ref{sec:background}, we collect background information on the normal Euler number and Lagrangian cobordisms, proving a generalization of Proposition~\ref{prop:tb-euler-filling} at the end.  Section~\ref{sec:obstruct} develops obstructions to Lagrangian fillings from normal rulings, including a discussion of canonical rulings and the definition of the resolution linking number.  We then proceed to examine the set of normal Euler numbers realized by Lagrangian fillings, proving Proposition~\ref{prop:finite-euler}.  Finally, Sections~\ref{sec:plus-ad}, \ref{sec:torus}, and \ref{sec:pretzel} examine the fillability of plus-adequate knots, torus knots, and $3$-stranded pretzel knots, respectively.

\subsection*{Acknowledgements}

We thank John Baldwin, John Etnyre, David Futer, Chuck Livingston, and Lisa Traynor for stimulating discussions.  The fourth author thanks the Institute for Advanced Study for hosting him during the final preparations of the paper; in particular, this material is based upon work supported by the Institute for Advanced Study.

\section{Background Notions}
\label{sec:background}


In this section, we recall basic notions about properly embedded non-orientable surfaces and Lagrangian cobordisms between Legendrian links.  We assume familiarity with the fundamentals of Legendrian knot theory in the standard contact $\rr^3$; see \cite{ cahn:intro, etnyre:knot-intro, geiges:intro, lisa:intro} for introductions.  The only new material in this section is contained in Proposition~\ref{prop:tb-euler}.

\subsection{The Normal Euler Number}
\label{ssec:euler}

In order to analyze non-orientable surfaces $(F, \partial F) \hookrightarrow ([0,1] \times Y, \{0,1\} \times Y)$ with null-homologous ends, we need an additional topological invariant first defined by Gordon and Litherland \cite{gl:signature}; see also \cite{batson:non-ori-slice, oss:unoriented}.  Any closed interval can stand in for $[0,1]$, and we will make such substitutions below without further comment.

Suppose that $F \subset [0,1] \times Y$ is a properly embedded surface with $\partial F = K_0 \sqcup K_1$ with $K_i \subset\{i\} \times Y$.  Let $F'$ be a small transverse pushoff of $F$ so that the pushoffs $K_i'$ at the ends  both realize the Seifert framing.  We compute the \dfn{(relative) normal Euler number $e(F)$} by finding compatible local orientations for $TF$ and $TF'$ at each intersection point in $F \cap F'$ --- which may be used, together with an ambient orientation on $[0,1] \times Y$, to assign a sign to each intersection --- and then adding up the contributions at the intersection points.

There are several equivalent ways of defining the relative normal Euler number.  Instead of specifying the framing for $F'$ at the ends, we may take Seifert surfaces $\Sigma_i$ for $K_i$, and consider the closed surface $\bar{F} = F \cup \Sigma_0 \cup \Sigma_1$.  We then compute the intersection number of $\bar{F}$ and a small transverse pushoff $\bar{F}'$ to get $e(F)$.  Another equivalent definition is to let $F'$ be the image of a section of the normal $S^1$ bundle of $F$ and let $K_0' \sqcup K_1' = \partial F'$; following \cite{gl:signature} and \cite[Lemma 4.2]{oss:unoriented}, we then see that $e(F) = \lk(K_0,K'_0)-\lk (K_1,K'_1)$.

\begin{rem}
If $F$ is orientable, then $e(F)=0$ since $\bar{F}$ represents the zero class in $H_2([0,1] \times Y; \zz)$.  Further, the normal Euler number is always even, as $\bar{F}$ always represents the zero class in $H_2([0,1] \times Y; \zz/2\zz)$. 
\end{rem}

As noted by Batson \cite{batson:non-ori-slice}, we may use the normal Euler number to refine the minimal $4$-dimensional crosscap number of a knot in $S^3$. The \dfn{$e$-crosscap number} $cr^e(K)$ of a knot $K$ is the minimal first Betti number of all properly embedded non-orientable surfaces $F \subset B^4$ with $\partial F = K$ and $e(F) = e$.  The $4$-dimensional \dfn{crosscap number} $cr_4(K)$ is simply the minimum value of $cr^e(K)$.

\subsection{Lagrangian Cobordisms}
\label{ssec:lagr}

The formal definition of a cylindrical-at-infinity Lagrangian cobordism between Legendrian submanifolds is the following.

\begin{defn}
\label{defn:cob}
	Let $\leg_\pm$ be Legendrian links in the contact manifold $(Y,\xi)$, where $\xi = \ker \alpha$.  An \dfn{(exact) Lagrangian cobordism} $L$ from $\leg_-$ to $\leg_+$ is an exact properly embedded Lagrangian submanifold of the symplectization $(\rr \times Y, d(e^t\alpha))$ satisfying:
	\begin{itemize}
	\item There exists $T_+ \in \rr$ such that $L \cap ([T_+,\infty) \times Y) = [T_+,\infty) \times \leg_+$;
	\item There exists $T_- < T_+$ such that $L \cap ((-\infty,T_-] \times Y) = (-\infty,T_-] \times \leg_-$; and
	\item The primitive of $(e^t\alpha)|_L$ is constant at each end of $L$.
	\end{itemize}
	The cobordism $L$ is a \dfn{filling} if $\leg_- = \emptyset$.
\end{defn}

The final condition in the definition of a cobordism is designed to allow cobordisms to be concatenated while preserving exactness; see \cite{chantraine:disconnected-ends} for a more thorough discussion.

There are three types of \dfn{elementary cobordisms} that are useful for constructing Lagrangian cobordisms:
\begin{description}
\item[Legendrian Isotopy] A Legendrian isotopy from $\leg_-$ to $\leg_+$ induces a Lagrangian cobordism from $\leg_-$ to $\leg_+$, though the construction is somewhat more complicated than simply taking the trace of the isotopy \cite{bst:construct, ehk:leg-knot-lagr-cob, eg:finite-dim}.

\item[$0$-Handle] Adding a disjoint, unlinked maximal Legendrian unknot $\Upsilon$ to $\leg$ induces a Lagrangian $0$-handle cobordism from $\leg$ to $\leg \sqcup \Upsilon$ \cite{bst:construct,ehk:leg-knot-lagr-cob}.

\item[$1$-Handle] Performing an ambient surgery between two cusps of the front projection of $\leg$ as in Figure~\ref{fig:1-handle} induces a Lagrangian $1$-handle cobordism \cite{bst:construct, rizell:surgery, ehk:leg-knot-lagr-cob}.  A $1$-handle attachment is  may be \dfn{oriented} or \dfn{unoriented} depending on the orientation of $\leg$ near the attaching regions; see the figure.  In practice, we will use the operation of ``pinching'' across the co-core of a $1$-handle as we move down a cobordism.
\end{description}

\begin{figure}
\centerline{\includegraphics{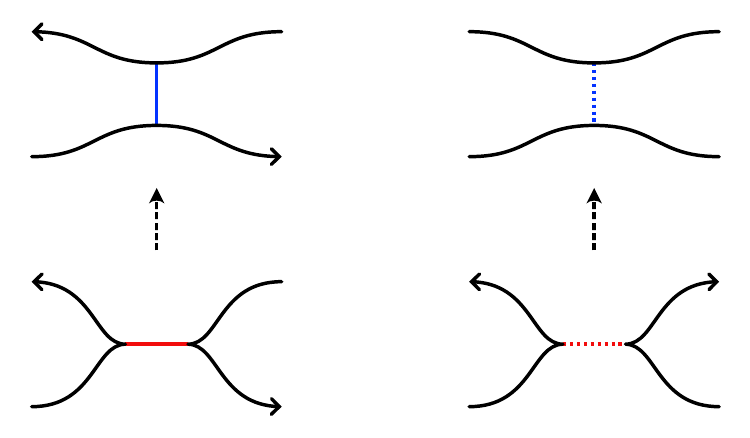}}
\caption{Attaching a $1$-handle along a Legendrian link:  (left) an oriented $1$-handle and (right) an unoriented $1$-handle.}
\label{fig:1-handle}
\end{figure}

A cobordism constructed by concatenating finitely many elementary cobordisms is called \dfn{decomposable}.  It is useful to introduce some notation:  denote the decomposition of $L$ into elementary cobordisms by $L=L_1 \odot \cdots \odot L_n$, with $L_i$ going from $\leg_{i-1}$ to $\leg_i$.  For more information about constructing Lagrangian cobordisms, see \cite{blllmppst:cob-survey}. 

\begin{ex} \label{ex:fig-8-filling}
	The figure-eight knot is non-orientably fillable by a Lagrangian Klein bottle with normal Euler number 4; see Figure~\ref{fig:fig-8}.  The normal Euler number may be computed combinatorially using Proposition~\ref{prop:tb-euler}, below.
\end{ex}

\begin{figure}
\begin{center}
\labellist
\small\hair 2pt
 \pinlabel {$1$-handles} [l] at 99 162
 \pinlabel {Isotopy} [l] at 99 45
\endlabellist
\includegraphics{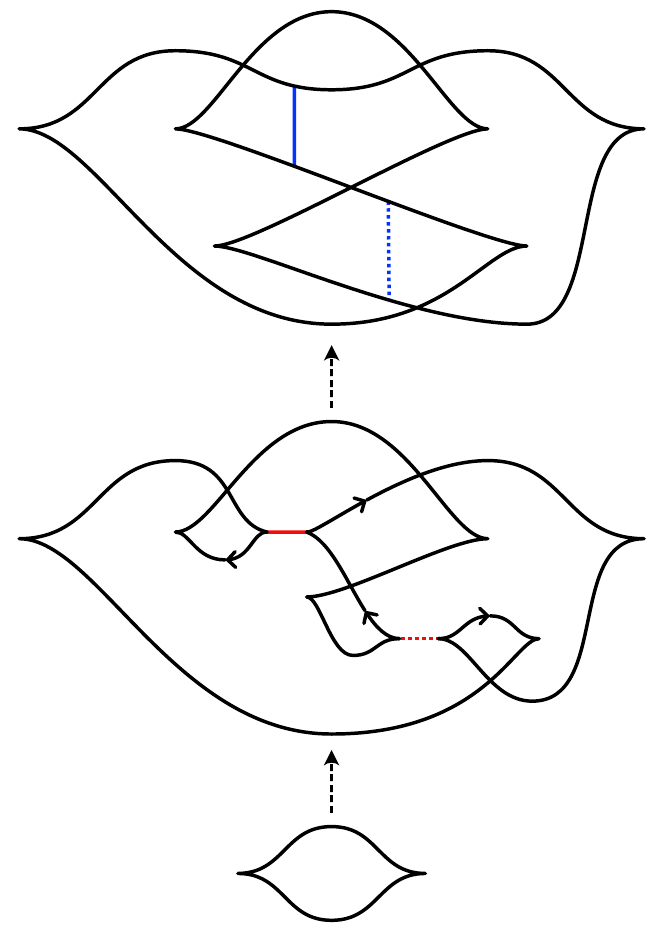}
\caption{A $0$-handle and isotopy produce the unknot in the middle. Two subsequent $1$-handles yield a filling of a Legendrian figure-eight knot by a Klein bottle with normal Euler number 4. }
\label{fig:fig-8}
\end{center}
\end{figure}

\begin{ex} \label{ex:8-21-filling}
	The $m(8_{21})$ knot is orientably fillable by a Lagrangian torus and non-orientably fillable by a Lagrangian Klein bottle with normal Euler number $0$; see Figure~\ref{fig:8-21}. 
\end{ex}

\begin{figure}
\begin{center}
\labellist
\small\hair 2pt
 \pinlabel {$1$-handles} [r] at 73 349
 \pinlabel {Isotopy} [r] at 73 229
 \pinlabel {$1$-handles} [l] at 288 350
 \pinlabel {Isotopy} [r] at 268 225
 \pinlabel {Isotopy} [r] at 268 178
 \pinlabel {$1$-handle} [l] at 292 113
 \pinlabel {Isotopy} [r] at 268 52
\endlabellist
\includegraphics{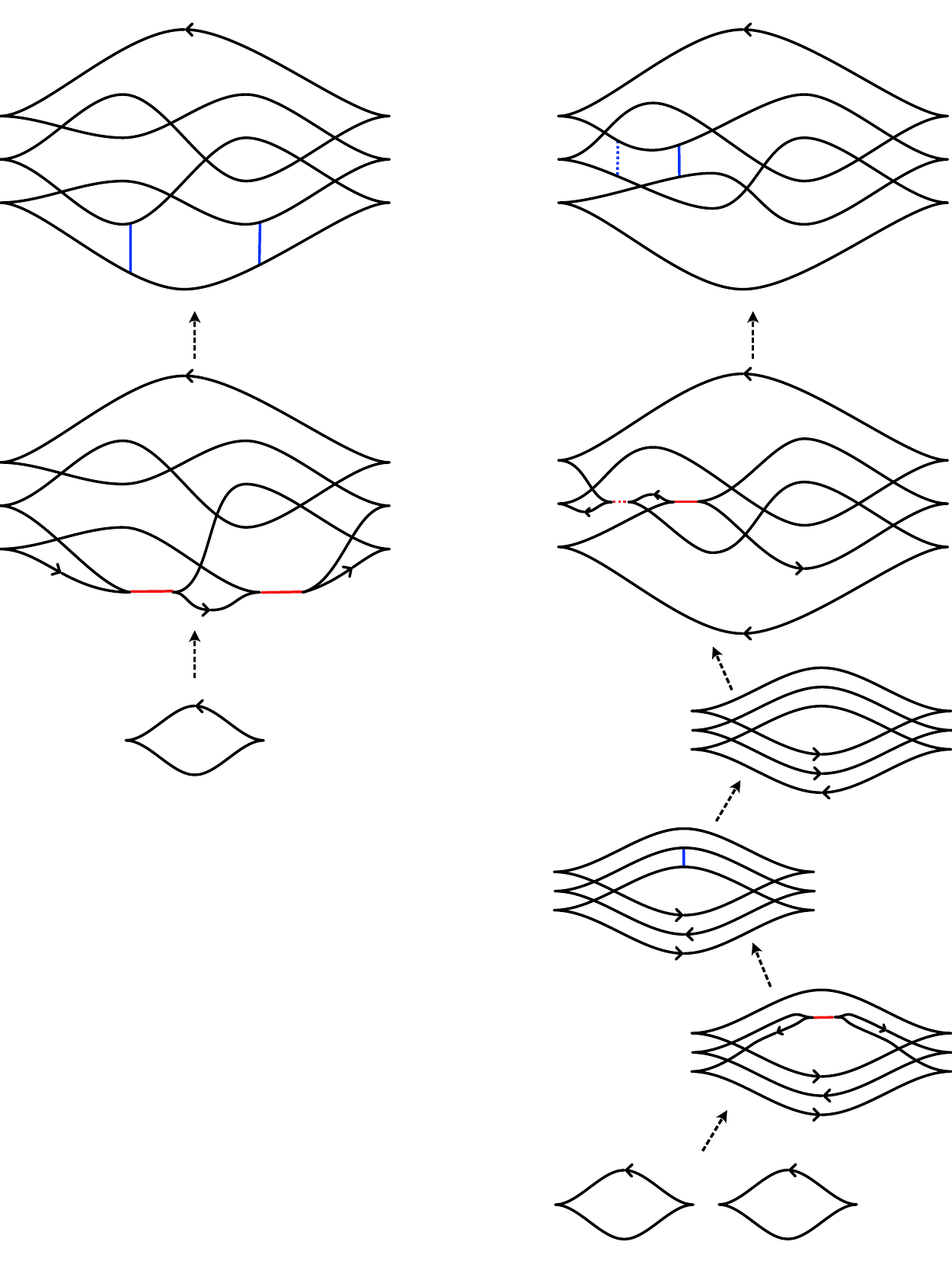}
\caption{Fillings of a Legendrian $m(8_{21})$ knot by a torus (left) and by a Klein bottle with normal Euler number 0 (right).  The middle isotopy at right is a cyclic permutation of the $3$-copy of the unknot \cite{kirill}. }
\label{fig:8-21}
\end{center}
\end{figure}

We finish this section by relating the normal Euler number, Euler characteristic, and Thurston-Bennequin numbers of a Lagrangian cobordism; Proposition~\ref{prop:tb-euler-filling} is an immediate corollary.

\begin{prop} \label{prop:tb-euler}
	If $L$ is a Lagrangian cobordism from $\leg_-$ to $\leg_+$ with normal Euler number $e(L)$, then
	\[ tb(\leg_+) - tb(\leg_-) = - \chi(L) - e(L). \]
\end{prop}

\begin{proof}
We prove that $e(L) = -tb(\leg_+) + tb(\leg_-)  - \chi(L)$ by calculating the normal Euler number from the definition.  It will prove useful to introduce an almost complex structure $J$ on the symplectization $\rr \times Y$ that leaves the contact planes invariant and sends the symplectization direction to the Reeb vector field $R_\alpha$.  Note that this choice of almost complex structure orients $\rr \times Y$ in a manner consistent with the symplectic structure.
	
	Consider a Morse function $f: L \to \rr$ that agrees with the symplectization coordinate $t$ outside of $[T_-,T_+] \times Y$.  Let $X$ be the gradient of $f$ (with respect to some metric).  Since $L$ is Lagrangian, the vector field $JX$ is normal to $L$ when nonzero, with $JX = R_\alpha$ outside of $[T_-,T_+] \times Y$.  Pushing $L$ off along $JX$ yields a surface $L'$ that intersects $L$ at the critical points of $f$ and has the contact (Thurston-Bennequin) framing at the ends.  We correct the framing of $L'$ at the ends by adjoining annuli that interpolate between $L' \cap \{T_\pm\} \times Y$ and pushoffs of $L \cap \{T_\pm \pm 1\} \times Y$ that realize the Seifert framing at each end; call the result $L''$.
	
	It remains to compute the signed intersection number of $L$ and $L''$.  The annuli contribute $-tb(\leg_+) + tb(\leg_-)$.  The remaining intersection points arise from the critical points of $f$.  A local computation shows that each such intersection point $p$ contributes $(-1)^{1+\ind_p f}$, which yields a total contribution of $-\chi(L)$ from the critical points of $f$. The result follows.
\end{proof}

\section{Obstructions to Lagrangian Fillings from Normal Rulings}
\label{sec:obstruct}


In this section, we use normal rulings of front diagrams to analyze decomposable Lagrangian fillings, both orientable and not. We begin with the foundational definitions of normal rulings, including their Euler characteristics and a new quantity that we term the \dfn{normal Euler number} of a ruling. We then recall that a decomposable filling yields a canonical ruling of the Legendrian link at the top \cite{atiponrat:obstruction}, and we connect the topology of a Lagrangian cobordism with the orientability, Euler characteristic, and normal Euler number of the canonical ruling.  Finally, we  attach a quantity called the \dfn{resolution linking number} to a ruling and show that it is invariant for rulings related by decomposable cobordism, hence yielding an obstruction to the existence of a decomposable filling.  In all of this work, the non-orientable setting is the more subtle because of the non-vanishing of the Euler number.

\subsection{Normal Rulings}
\label{ssec:rulings}

A normal ruling is a combinatorial structure on the front diagram of a Legendrian link inspired by the theory of generating families; see \cite{ruling-survey} for a broader overview.  Essentially, a ruling is a decomposition of a front diagram of a Legendrian $\leg$ into a set of disks (called ``ruling disks"), each of which is planar isotopic to a maximal Legendrian unknot, with additional restrictions to control the interaction of the disks where they meet. In light of Theorem~\ref{thm:ruling-poly}, below, we will conflate notation for a Legendrian knot and its front diagram in this paper.

To define a \dfn{normal ruling} on $\leg$, we assume that the $x$ coordinates of all crossings and cusps are distinct.  A normal ruling consists of a set $\rho$ of crossings of $\leg$, called \dfn{switches}, that satisfies a set of combinatorial conditions.  To elucidate those conditions, let $\leg^\rho$ denote a new Legendrian link obtained from $\leg$ by resolving the switches of $\rho$ into horizontal line segments as in Figure~\ref{fig:ruling-defn}(a).  We say that the components of $\leg^\rho$ to which the new horizontal line segments belong are \dfn{incident} to the switch.  The components of $\leg^\rho$ must satisfy the following three conditions:
\begin{enumerate}
\item Each component of $\leg^\rho$ is planar isotopic to the standard diagram of the maximal Legendrian unknot.  In particular, each component bounds a \dfn{ruling disk} in the plane.
\item Exactly two components are incident to each switch.
\item Inside a small vertical strip around each switch, the ruling disks incident to the crossing are either nested or disjoint; see Figure~\ref{fig:ruling-defn}(b).
\end{enumerate}

\begin{figure}
\centerline{\includegraphics{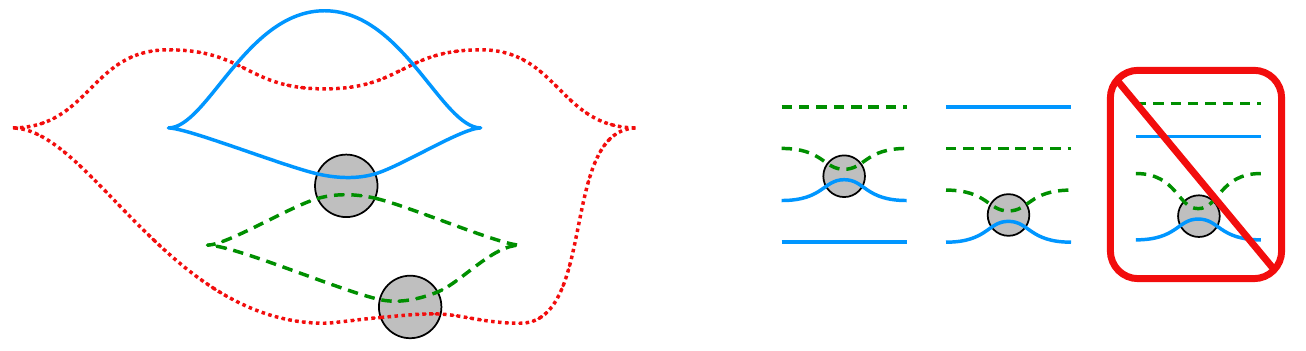}}
\caption{(Left) The resolution of switches in a ruling of a figure-eight knot.  (Right) Possible configurations of ruling disks at a switch in a ruling, up to reflection about the horizontal.  After this figure, we will no longer draw rulings with resolved crossings.}
\label{fig:ruling-defn}
\end{figure}

A ruling $\rho$ is \dfn{oriented} if all of its switches are positive crossings and \dfn{unoriented} otherwise.

We attach two quantities to a normal ruling, the first a well-known analogue of the Euler characteristic and the second a novel analogue of the normal Euler number.  We first set some notation. Denote by $c(\leg)$ the number of right cusps.  Let $s_+(\rho)$ (resp.\ $s_-(\rho)$) be the number of positive (resp.\ negative) switches in $\rho$, with $s(\rho) = s_+(\rho) + s_-(\rho)$.  Finally, fix orientations $o$ on $\leg$ and $o^\rho$ on $\leg^\rho$ and consider the crossings of $\leg^\rho$ where the sign of the crossing coming from $o$ differs from the sign coming from $o^\rho$; we refer to those crossings as \dfn{flipped crossings} and the sections of $\leg^\rho$ on which $o$ and $o^\rho$ disagree the \dfn{flipped region} or \dfn{flipped strands}.  Let $f_+(\rho, o, o^\rho)$ (resp.\ $f_-(\rho, o, o^\rho)$) be the number of flipped crossings of $\leg^\rho$ with positive (resp.\ negative) sign with respect to $o^\rho$.  See Figure~\ref{fig:s-delta} for an illustration of this notation.

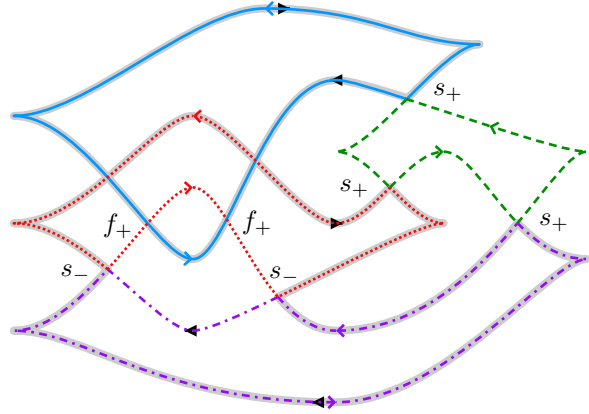
\begin{figure}
\labellist
\small\hair 2pt
 \pinlabel {$s_+$} [l] at 164 126
 \pinlabel {$s_+$} [r] at 144 90
 \pinlabel {$s_+$} [l] at 204 77
 \pinlabel {$s_-$} [r] at 39 58
 \pinlabel {$s_-$} [b] at 110 52
 \pinlabel {$f_+$} [r] at 55 76
 \pinlabel {$f_+$} [l] at 92 77
\endlabellist
\centerline{\includegraphics{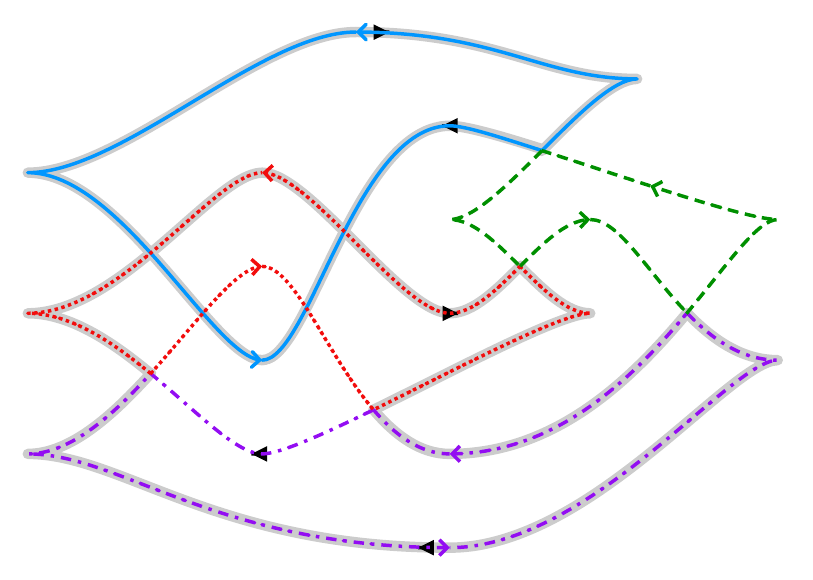}}
\caption{For this Legendrian $\leg$, ruling $\rho$, and orientations $o$ (solid arrows) and $o^\rho$ (open arrows), we have $s_+(\rho) = 3$, $s_-(\rho) = 2$, $f_+(\rho, o, o^\rho) = 2$, and $f_-(\rho, o, o^\rho) = 0$. The flipped region is shaded. Thus, we have $\chi(\rho) = -1$ and $e_2(\rho) \equiv 0$. }
\label{fig:s-delta}
\end{figure}

We are now ready to define the Euler numbers of a ruling.

\begin{defn} [\cite{chv-pushkar}]
The \dfn{Euler characteristic} of a ruling $\rho$ is an integer defined by
\[ \chi(\rho) = c(\leg) - s(\rho).\]
\end{defn}

K\'alm\'an \cite{kalman:spanning} observed that $\chi(\rho)$ is the Euler characteristic of the (possibly immersed) \dfn{ruling surface} $\Sigma_\rho$ constructed by connecting the ruling disks of $\rho$ by bands at the switches.

\begin{defn}
The \dfn{normal Euler number} of a ruling $\rho$ and orientations $o$ and $o^\rho$ is the element of $\zz/2\zz$ defined by
\[ e_2(\rho) \equiv s_-(\rho) + f_+(\rho,o,o^\rho) - f_-(\rho,o,o^\rho)  \mod 2.\]
\end{defn}

It follows immediately from the definition that if $\rho$ is orientable, then $e_2(\rho) = 0$.  To justify the notation $e_2(\rho)$, we prove the following lemma.

\begin{lem} \label{lem:e-independent}
	The normal Euler number of a ruling does not depend on the orientations $o$ and $o^\rho$.
\end{lem}

\begin{proof}
	Since $s_-(\rho)$ does not depend on the choices of orientation, we need only prove that $f_+(\rho,o,o^\rho) - f_-(\rho,o,o^\rho)$ does not depend on $o$ and $o^\rho$ modulo $2$.  In fact, it suffices to prove that $f_+(\rho,o,o^\rho) - f_-(\rho,o,o^\rho)$ is invariant under a change of orientation of a single component of the link $\leg^\rho$. 
	
	Denote by $\pi^\rho$ the orientation of $\leg^\rho$ obtained by switching the orientation on a single component $\leg^\rho_0 \subset \leg^\rho$.  Let $\lk_0 (o^\rho)$ denote the sum of the linking numbers between $\leg^\rho_0$ and all of the other components of $\leg^\rho$ with respect to the orientation $o^\rho$.  If we label the flipped crossings of $\leg_0^\rho$ by $f^0_\pm(\rho,o,o^\rho)$ and the non-flipped crossings of $\leg_0^\rho$ by $\bar{f}^0_\pm(\rho,o,o^\rho)$, then we may compute
\begin{equation} \label{eq:lk0}
	\lk_0(o^\rho) = \frac{1}{2}\left(f^0_+(\rho,o,o^\rho) + \bar{f}^0_+(\rho,o,o^\rho) - f^0_-(\rho,o,o^\rho) - \bar{f}^0_-(\rho,o,o^\rho)  \right).
	\end{equation}
	
	Further, notice that passing from $o^\rho$ to $\pi^\rho$ swaps the sets of flipped and non-flipped crossings on $\leg_0^\rho$, and also reverses the signs of those crossings.  That is, we obtain 
\begin{equation} \label{eq:f0}
	\begin{split} 
		f^0_+(\rho,o,\pi^\rho) &= \bar{f}^0_-(\rho,o,o^\rho) \\
		f^0_-(\rho,o,\pi^\rho) &= \bar{f}^0_+(\rho,o,o^\rho)
	\end{split}
\end{equation}
	
	Combining Equations~\eqref{eq:lk0} and \eqref{eq:f0} yields the following computation:
\[ e_2(\rho,o,o^\rho) - e_2(\rho, o, \pi^\rho) =  2 \lk_0 (o^\rho) \equiv 0 \mod 2,\]
which completes the proof of the lemma.
\end{proof}

We may organize the set of rulings of a front diagram into the \dfn{ruling polynomial} \cite{chv-pushkar}: 
\[R_\leg(z) = \sum_{\text{Rulings } \rho} z^{1-\chi(\rho)}.\]
The \dfn{oriented ruling polynomial} $R_\leg^o(z)$ is defined similarly by summing over oriented rulings.  The ruling polynomials $R_\leg$ and $R_\leg^0$ are Legendrian --- in fact, smooth --- invariants, thus justifying our conflation of a Legendrian knot and its front diagram in our notation.\footnote{One can upgrade the oriented ruling polynomial to a \emph{graded} ruling polynomial, which is an effective Legendrian (as opposed to smooth) invariant \cite{chv-pushkar}.}

\begin{thm}[\cite{rutherford:kauffman}]
\label{thm:ruling-poly}
	Given a Legendrian link $\leg$ with front diagram $D$, 
	\begin{enumerate}
	\item The ruling polynomial $R_\leg(z)$ is the coefficient of $a^{-tb(\leg)-1}$ in the Kauffman polynomial $F_\leg(a,z)$. 
	\item The oriented ruling polynomial $R_\leg^o(z)$ is the coefficient of $a^{-tb(\leg)-1}$ in the HOMFLY polynomial $P_\leg(a,z)$. 
	\end{enumerate}
\end{thm}

Combining this result with the upper bounds on the Thurston-Bennequin invariant from the Kauffman and HOMFLY polynomials, we see that if a Legendrian knot $\leg$ has a front diagram with a ruling, then it must maximize $tb$ \cite{rutherford:kauffman}.

\subsection{Fillings and the Existence of Rulings}
\label{ssec:augm-ruling}

A fundamental link between rulings and fillings uses the machinery of Legendrian contact homology, a Floer-type invariant of Legendrian submanifolds.

\begin{prop}
	\label{prop:ruling-obstr}
	If a smooth knot $K$ is fillable, then the Kauffman bound on the maximal Thurston-Bennequin number is sharp and every Legendrian representative with maximal Thurston-Bennequin number has a ruling.
\end{prop}

\begin{proof} 
Suppose that $K$ has a Legendrian representative $\leg$ with a Lagrangian filling $L$.  The Legendrian contact homology DGA of such a Legendrian has an augmentation \cite{rizell:lifting,ekholm:lagr-cob}.  That augmentation, in turn, yields a ruling of $\leg$ \cite{fuchs-ishk, rulings}.  Theorem~\ref{thm:ruling-poly} then implies both parts of the conclusion.
\end{proof}

\begin{rem}
	If the filling of $K$ is orientable, then the augmentations and rulings are $2$-graded, and hence the ruling is also orientable.
\end{rem}

\subsection{Canonical Rulings for Decomposable Fillings}
\label{ssec:canonical}
 
In this section, we begin to explore a more subtle relationship between rulings and decomposable fillings. Atiponrat \cite[Lemma 2]{atiponrat:obstruction} proved that a decomposable filling $L$ of a Legendrian $\leg$ induces a canonical ruling $\rho_L$ on $\leg$; Pan also commented on this fact in \cite[\S5.5]{pan:aug-category-cob} and noted that the proof extends to cobordisms from $\leg_-$ to $\leg_+$ with a given ruling on $\leg_-$. The proof, in essence, comes from Chekanov and Pushkar's proof that rulings --- even orientable rulings --- are invariant under Legendrian isotopy \cite{chv-pushkar} and the fact that $0$-handles create and $1$-handles merge ruling disks; see Figure~\ref{fig:canonical-ruling} for an illustration of the latter. We note, however, that the canonical ruling of a decomposable filling may not correspond with the ruling produced by Proposition~\ref{prop:ruling-obstr}, though the canonical ruling construction does yield an alternative --- and more elementary --- proof of the proposition in the case that the filling is decomposable.

The main result in this section is to refine the construction of a canonical ruling to relate the orientability of the cobordism $L$ with the orientability of the ruling $\rho_L$. 

\begin{figure}
	\centerline{\includegraphics{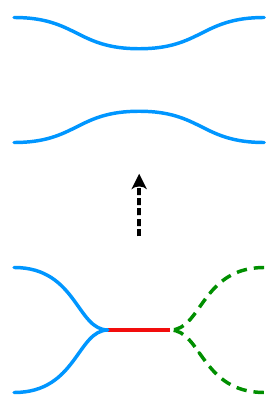}}
  \caption{Constructing a canonical ruling across a $1$-handle.}
  \label{fig:canonical-ruling}
\end{figure}

\begin{prop}
\label{prop:ori-canonical-ruling}
	Suppose $L$ is a decomposable cobordism from $\leg_-$ to $\leg_+$, that $\leg_+$ is connected, and that $\rho$ is an orientable ruling of $\leg_-$.  The canonical ruling $\rho_L$ of $\leg_+$ is orientable if and only if $L$ is.
\end{prop}

\begin{proof}
We use the notation $L = L_1 \odot \cdots \odot L_n$.  First note that since $\leg_+$ is connected, the lack of $2$-handles in the decomposition of $L$ implies that $L_i \odot \cdots \odot L_n$ is connected for any $i$.

The key observation is that $L$ is orientable if and only if all of the $1$-handles used to construct $L$ are  oriented.  Indeed, let $L_i$ be the cobordism induced by the topmost non-orientable $1$-handle. Since there are no $2$-handles and $\leg_+$ is connected, we see that $L_{i+1} \odot \cdots \odot L_n$ is connected. Thus, we may form a curve $\gamma$ in $L_{i+1} \odot \cdots \odot L_n$ connecting the two points in the cosphere. As $L_{i+1} \odot \cdots \odot L_n$ is orientable, a neighborhood $N$ of $\gamma$ is a rectangle with proper boundary oriented by $\leg_i$.  Attaching the co-core of the non-oriented $1$-handle to $N$ along $\partial N \cap \leg_i$ yields a M\"obius strip.  Thus, $L_{i} \odot \cdots \odot L_n$ is non-orientable, and hence $L$ is as well. The reverse direction is obvious.

We are now ready to prove the proposition. Suppose that $L$ is orientable. Since Legendrian isotopy and the addition of a $0$-handle preserve orientability of the canonical ruling, we need only consider $1$-handles. The observation above shows that the $1$-handles in $L$ are all oriented.  It is clear that the canonical ruling procedure in Figure~\ref{fig:canonical-ruling} creates an orientable ruling on $\leg_{i}$ from an orientable ruling on $\leg_{i-1}$. Thus, if $L$ is orientable, then $\rho_L$ is orientable if $\rho$ is.

Conversely, suppose that $\rho_L$ is orientable.  Reversing the procedure for extending a canonical ruling across a $1$-handle entails pinching across a ruling disk.  If a ruling disk is oriented, then the Legendrian is oriented in opposite directions along the top and bottom strands of that ruling disk.  Thus, we must have that every $1$-handle in $L$ is oriented.  Thus, we see that $L$ is orientable.
\end{proof}

Proposition~\ref{prop:ori-canonical-ruling} is an effective obstruction to the existence of a decomposable non-orientable filling.  We encapsulate the obstruction in the following corollary, which follows from the idea that a non-orientable decomposable filling of $\leg$ induces a non-orientable ruling on $\leg$, which, in turn, is counted by the ruling polynomial but not the oriented ruling polynomial.

\begin{cor}
\label{cor:ruling-poly-obstr}
	If $\leg$ has a non-orientable decomposable filling, then 
	\[ R_\leg(z) \neq R^o_\leg(z).\]
\end{cor}

\begin{ex}
	We saw in Example~\ref{ex:fig-8-filling} that the figure-eight knot has a non-orientable decomposable filling.  As promised by the corollary, one can compute that $R_\leg (z) \neq R^o_\leg(z)$ and hence that the relevant coefficients of the HOMFLY and Kauffman polynomials for the figure-eight knot differ.
\end{ex}

\begin{rem} \label{rem:ruling-homfly-k}
Since the condition on the ruling polynomials in Corollary~\ref{cor:ruling-poly-obstr} can be read off of the HOMFLY and Kauffman polynomials of the underlying smooth knot, it follows that the obstruction to the existence of a decomposable non-orientable filling does not depend on the (maximal) Legendrian representative.
\end{rem}

\subsection{Canonical Rulings and Euler Numbers}

In this section, we relate the topology of a decomposable Lagrangian cobordism to the Euler numbers of the corresponding canonical ruling.

\begin{prop}
\label{prop:canonical-chi}
	If $L$ is a decomposable Lagrangian cobordism from $\leg_-$ to $\leg_+$, $\rho$ a ruling of $\leg_-$, and $\rho_L$ the canonical ruling of $\leg_+$ induced by $L$, then
	\[\chi(\rho_L) - \chi(\rho) = \chi(L).\]
\end{prop}

\begin{proof}
	We check the relation for each of the elementary cobordisms. The equation holds for cobordisms induced by Legendrian isotopies since $\chi(\rho)$ is a Legendrian isotopy invariant and the Euler characteristic of a cylinder vanishes.  The equation holds for the addition of a $0$-handle since the new maximal unknot raises $\chi(\rho)$  by $1$ as we get a new ruling disk with no new switches; this matches the Euler characteristic of a disk.  Finally, the equation holds for the addition of a $1$-handle, as $\chi(\rho)$ changes by $-1$ with the loss of a right cusp and no change to the switches; this matches the Euler characteristic of a pair of pants.
\end{proof}

\begin{prop}
\label{prop:canonical-e}
Let $L$ be a decomposable Lagrangian cobordism from $\leg_-$ to $\leg_+$, $\rho$ a ruling of $\leg_-$, and $\rho_L$ the canonical ruling of $\leg_+$ induced by $L$.  We compute that
	\[e_2(\rho_L) - e_2(\rho) \equiv \frac{1}{2} e(L) \mod 2.\]
\end{prop}

\begin{proof}
	We check the relation for each of the elementary cobordisms, this time working one Reidemeister move at a time for cobordisms induced by Legendrian isotopy. In these cases, we use the ruling correspondence laid down in \cite{chv-pushkar}; see also \cite{atiponrat:obstruction}. Note that for all but the $1$-handle case, the normal Euler number of the cobordism vanishes, so in those cases we must prove that $e_2$ does not change.
	
	\textbf{Reidemeister 0}.  If neither or both crossings are switched, or there are more than two ruling disks involved, then the rulings before and after an R0 move are combinatorially identical.  Thus, we need only consider the case where there is a single switch and only two ruling disks appear locally, as in Figure~\ref{fig:r0+}.  
	
	We claim that we can choose orientations on the two ruling disks that satisfy the following three conditions on both sides of the R0 move:

	\begin{enumerate}
		\item Exactly one flipped strand passes smoothly through any negative switch or negative crossing; 
		\item Exactly zero or two flipped strands pass through any positive switch or positive crossing; and
		\item The flipped strands are identical away from the local picture on both sides of the R0 move.
	\end{enumerate}

	In fact, one such choice is to orient the two ruling disks counterclockwise, as confirmed for configurations where the upper left switch is positive in Figures~\ref{fig:r0+}; a similar check can be made in the case that the upper left switch is negative. With such choices in hand, it is straightforward to see that a negative switch on one side of the R0 move is paired with a $f_+$ crossing on the other, thus balancing the contributions to $e_2$; there are no $f_-$ crossings to balance by our choices.
	
\begin{figure}
\labellist
\small\hair 2pt
 \pinlabel {$f_+$} [t] at 210 71
 \pinlabel {$s_-$} [t] at 260 71
 \pinlabel {$s_-$} [t] at 42 10
 \pinlabel {$f_+$} [t] at 90 10
\endlabellist
\centerline{\includegraphics{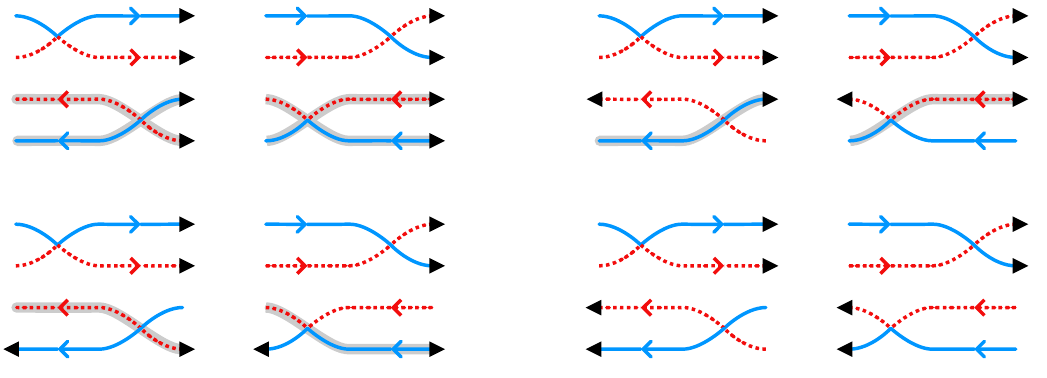}}
\caption{Choices of orientation for $\leg^\rho$ around an R0 move that satisfy the three conditions in the proof when the upper left switch is positive.  Similar choices may be made with the upper left switch is negative.}
\label{fig:r0+}
\end{figure}
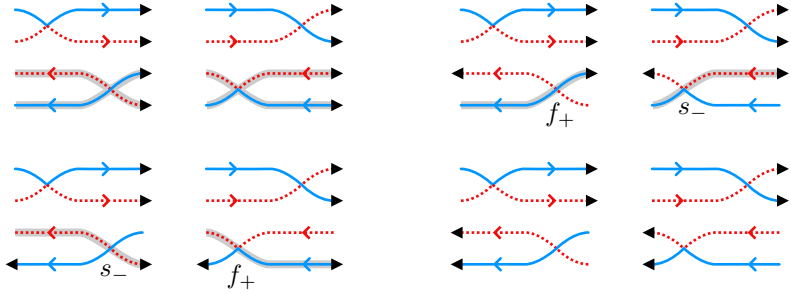

	\textbf{Reidemeister I}.  There are no new negative switches or flipped crossings in this case, so $e_2$ does not change.  See the left side of Figure~\ref{fig:r1r20h}.

\begin{figure}
\labellist
\small\hair 2pt
 \pinlabel {$\emptyset$} at 193 26
\endlabellist
\centerline{\includegraphics{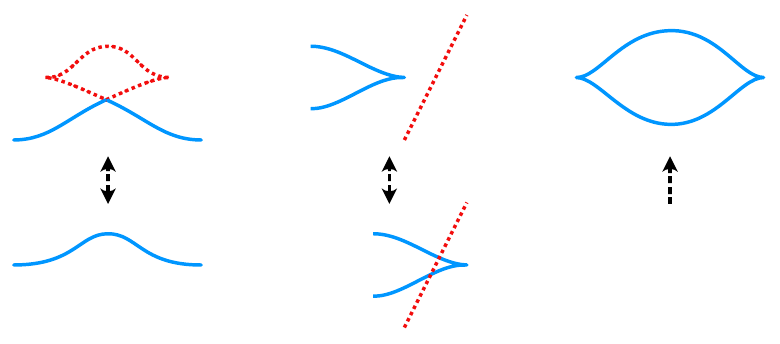}}
\caption{The normal Euler numbers of the corresponding rulings are equal across (left) an RI move (the switch is positive), (center) an RII move, or (right) a $0$-handle.}
\label{fig:r1r20h}
\end{figure}

	\textbf{Reidemeister II}.  There are no new negative switches in this case, and any flipped crossings arise in oppositely-signed pairs.  Thus, $e_2$ does not change. See the center of Figure~\ref{fig:r1r20h}.
		
	\textbf{Reidemeister III}.  If no crossing or all three crossings are switched, then the local contributions to $e_2$ before and after the RIII move are identical.
	
	If a single crossing is switched and is positive, then we may orient the components of $\leg_-^\rho$ so that no flipped regions appear in the local picture of the RIII move.  It follows that there are no local contributions to $e_2$ before and after the RIII move.  If the single switch is negative, then we may orient the components of $\leg_-^\rho$ so that exactly one flipped strand passes smoothly through the negative switch and no other strands are flipped. In this case, there is one flipped crossing that has the same sign before and after the RIII move; see Figure~\ref{fig:r3-1}.
	
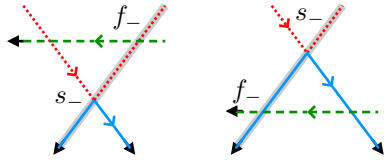
\begin{figure}
\labellist
\small\hair 2pt
 \pinlabel {$s_-$} [r] at 31 28
 \pinlabel {$f_-$} [r] at 53 59
 \pinlabel {$f_-$} [r] at 98 31
 \pinlabel {$s_-$} [b] at 114 54
\endlabellist
\centerline{\includegraphics{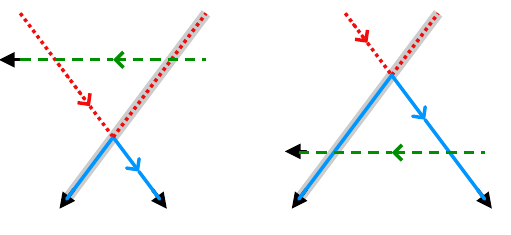}}
\caption{One case of an RIII move with one negative switch. There are choices of orientations on the components of $\leg^\rho$ so that there is one smooth flipped strand and the flipped crossings have the same signs before and after the move.}
\label{fig:r3-1}
\end{figure}

	If two crossings are switched, then as in the R0 cases, we may choose orientations on the components of $\leg_-^\rho$ so that the flipped strands are identical away from the local picture on both sides of the RIII move and so that local contributions to $e_2$ match; see Figure~\ref{fig:r3-2}.  
	
	Thus, for any RIII move, $e_2$ does not change.
	
\begin{figure}
\labellist
\small\hair 2pt
 \pinlabel {$s_-$} at 21 31
 \pinlabel {$s_-$} at 106 54
 \pinlabel {$s_-$} at 127 54
 \pinlabel {$f_-$} at 127 22
 \pinlabel {$s_-$} at 226 158
 \pinlabel {$s_-$} at 284 161
 \pinlabel {$s_-$} at 218 61
 \pinlabel {$s_-$} at 197 31
 \pinlabel {$f_-$} at 237 31
 \pinlabel {$s_-$} at 303 54
\endlabellist
\centerline{\includegraphics{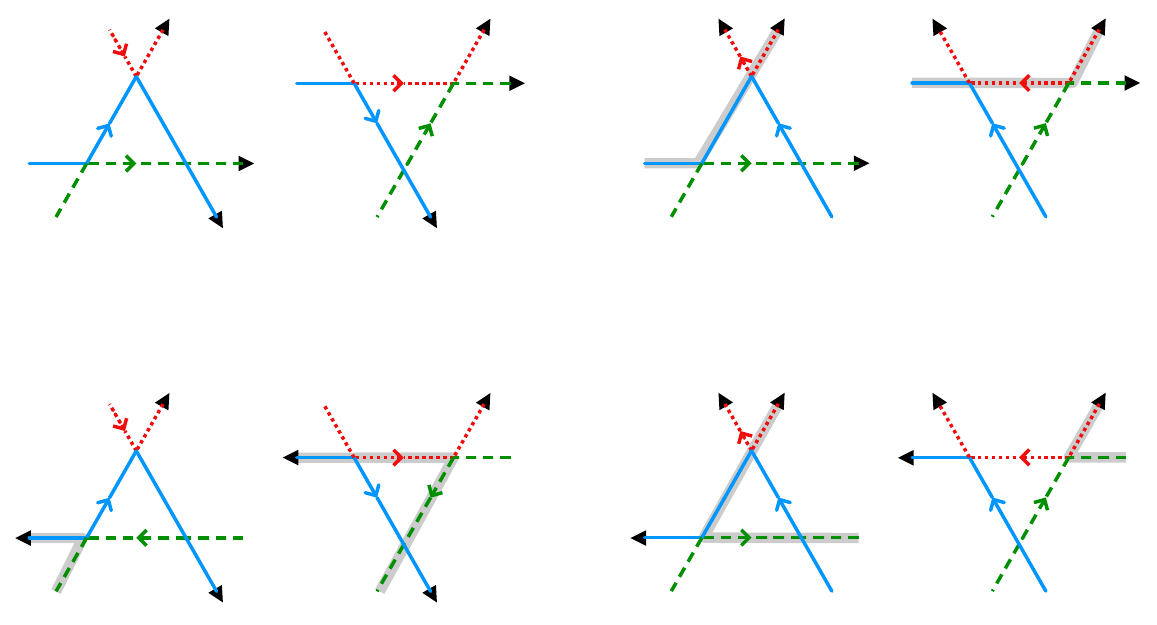}}
\caption{For each possible orientation (up to an overall reversal) of the strands of an RIII move with two switches, we display orientations on the components of $\leg^\rho$ so that the flipped strands are identical away from the local picture and so that local contributions to $e_2$ match.  In particular, note that $s_-$ and $f_-$ cancel.}
\label{fig:r3-2}
\end{figure}
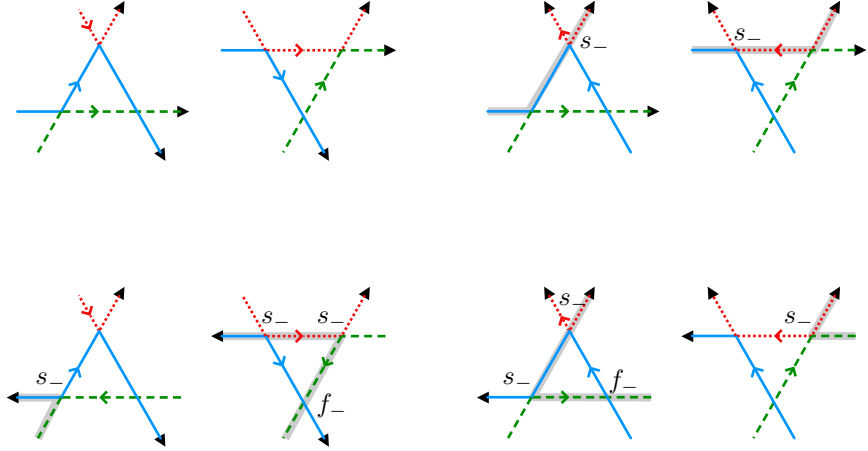

	\textbf{0-handle}. There are no new negative switches or flipped crossings in this case, so $e_2$ does not change. See Figure~\ref{fig:r1r20h}(c).

	\textbf{1-handle}.	In this case, Proposition~\ref{prop:tb-euler} shows that $e(L) = tb(\leg_-) - tb(\leg_+) + 1$.  
	 
	If the $1$-handle is orientable, then the orientations on $\leg_-$ and $\leg_+$ match outside a neighborhood of the $1$-handle, and hence so do the signs of all of the crossings. Since the signs of the crossings are unchanged but $\leg_+$ has one fewer cusp than $\leg_-$, we have $e(L) = 0$. On the other hand, we may choose orientations on the resolved diagram $\leg_-^\rho$ so that the cusps involved in the $1$-handle are either both flipped or both not flipped; choose the corresponding orientation on $\leg_+^{\rho_L}$.  Since the signs of the switches are the same before and after the $1$-handle, as are the signs of the crossings in the flipped regions, we see that $e_2$ is also unchanged. 
	
	If the $1$-handle is not orientable, then a closer analysis is needed.  For any choice of orientation of $\leg_+$, there is a ``reversing strand'' along which the orientation of $\leg_+$ and that of $\leg_-$ disagree; see the example in Figure~\ref{fig:e-1h}. Along the reversing strand, flipped regions and crossing signs (at least at those crossings for which exactly one strand is the reversing strand) are opposite for $\leg_-$ and $\leg_+$. 
	
\begin{figure}
\centerline{\includegraphics{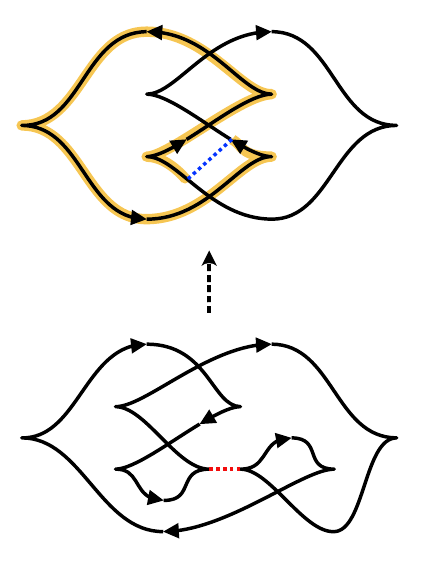}}
\caption{A non-orientable $1$-handle induces a ``reversing strand'' (highlighted at top) along which the orientations of $\leg_-$ and $\leg_+$ disagree, and hence along which flipped regions and crossing signs (at least at those crossings for which exactly one strand is the reversing strand) are reversed. \js{Is this useful?}}
\label{fig:e-1h}
\end{figure}
	
	In the expression for $e(L)$, we see that the elimination of a right cusp in $\leg_+$ cancels the $+1$.  Thus, we need only compare how changes in crossing signs  change $tb$ and $e_2$.  Suppose that the reversing strand changes a crossing from positive to negative.  On one hand, we may compute the local contribution to the change in $tb$ to be $\frac{1}{2}(tb(\leg_-) + tb(\leg_+)) = 1$. On the other hand, if the crossing is a switch, then $s_-(\rho)$ increases by $1$ and hence so does $e_2$.  If the crossing contributes to $f_-$, then the reversing strand unflips the crossing, and hence $f_-$ decreases by $1$, implying that $e_2$ increases by $1$. Note that such a crossing cannot contribute to $f_+$. Finally, if the crossing is not flipped, then the reversing strand flips it to contribute to $f_+$, implying that $e_2$ increases by $1$.  A similar analysis holds if the crossing changes from negative to positive.  In all cases, the change in $e_2$ matches that of $\frac{1}{2} e(L)$.
\end{proof}

\subsection{The Resolution Linking Number}
\label{ssec:rln}

Not every ruling need be the canonical ruling of a Lagrangian filling.  Atiponrat \cite{atiponrat:obstruction} defined one  obstruction, the parity of the number of ``clasps'' in a ruling that, loosely speaking, quantifies how far the ruling surface is from being a ribbon surface.  In this section, we define a more computable obstruction based on the linking between the boundaries of ruling disks.  We suspect that Atiponrat's parity and the unoriented resolution linking number contain the same information.

\begin{defn}
\label{defn:rlk}
	Let $\rho$ be an orientable ruling on a Legendrian link $\leg$. The \dfn{resolution linking number} $\rlk(\rho)$ is defined to be the sum of the pairwise linking numbers of the oriented link $\leg^\rho$.  For any ruling $\rho$ of $\leg$, the \dfn{unoriented resolution linking number} $\rlk_2(\rho)$ is defined to be the mod $2$ reduction of the sum of the pairwise linking numbers of the link $\leg^\rho$ with any chosen orientation of the components.
\end{defn}

Note that $\rlk_2$ is well-defined, as changing the orientation of a single component $\leg_i$ of $\leg^\rho$ would negate the linking numbers between $\leg_i$ and the other components, but those negations do not change the parity of the result.

It turns out that the resolution linking numbers are combinations of familiar invariants.  The first item in the lemma is due to \cite{liu-zhou}.

\begin{lem}
\label{lem:rlk-alt}
	\begin{enumerate}
	\item The oriented resolution linking number satisfies	
	\begin{equation} \label{eq:rlk}
	\rlk(\rho) = \frac{1}{2}(tb(\leg) + \chi(\rho)).
	\end{equation}
	\item For any choice of orientations  $\leg$ and of $\leg^\rho$, the unoriented resolution linking number satisfies
	\begin{equation} \label{eq:rlk2}
	\rlk_2(\rho) \equiv \frac{1}{2}\left(tb(\leg) + \chi(\rho) + 2 e_2(\rho) \right)\mod 2.
	\end{equation}
	\end{enumerate}
\end{lem}

We see immediately that the set of all oriented resolution linking number values of a Legendrian $\leg$ is a topological quantity.

\begin{cor}
\label{cor:rlk-homfly}
	The set of all oriented resolution linking numbers for a Legendrian $\leg$ is determined by the exponents of the variable $z$, shifted by $tb(\leg)$, in the coefficient of $a^{tb(\leg)-1}$ in the HOMFLY polynomial.
\end{cor}

\begin{proof}[Proof of Lemma~\ref{lem:rlk-alt}]
Generalizing notation from Lemma~\ref{lem:e-independent}, denote by $\bar{f}_\pm$ the non-flipped crossings of $\leg^\rho$. The proof of (2) follows from adding up the following identities for some choices of orientation $o$ and $o^\rho$, and then slightly generalizing the proof of Lemma~\ref{lem:e-independent} to see that $\rlk_2(\rho) = \frac{1}{2}\left( f_+ + \bar{f}_+ - f_- - \bar{f}_- \right)$:

\begin{align*}
	tb(\leg) &= (\bar{f}_+ + f_- + s_+(\rho)) - (\bar{f}_- + f_+ + s_-(\rho))  - c(D) \\
	\chi(\rho) &= c(D) - s_+(\rho) - s_-(\rho) \\
	2e_2(\rho) &= 2(s_-(\rho) + f_+ - f_-)
\end{align*}

The proof of (1) follows the same computation, this time noting that all of the steps work over the integers, and that the quantities $s_-(\rho)$ and $f_\pm(\rho)$  vanish. 
\end{proof}

If $\leg$ has an orientable decomposable filling $L$, then, on one hand, we know $\tb(\leg) = -\chi(L)$, while on the other, we have $\chi(\rho_L) = \chi(L)$.  Thus, when $\rho$ is the canonical ruling for a filling $L$, the first formula in the lemma above shows that $\rlk(\rho)$ vanishes. The contrapositive of this string of ideas shows that the resolution linking number is an obstruction to the existence of an orientable decomposable filling.  A similar line of reasoning applies to non-orientable fillings and $\rlk_2$. This perspective leads to Theorem~\ref{thm:rlk}.

\begin{thm}
\label{thm:rlk}
	If $L$ is an orientable decomposable cobordism from $\leg_-$ to $\leg_+$ and that $\rho$ is an oriented ruling of $\leg_-$, then
	\begin{equation} 
	 \rlk(\rho_L) = \rlk (\rho).
	 \end{equation}
	The equation holds modulo $2$ for non-orientable decomposable cobordisms and rulings with $\rlk_2$ in place of $\rlk$.
\end{thm}

\begin{proof}
	The proof follows directly from Lemma~\ref{lem:rlk-alt} and Propositions \ref{prop:tb-euler}, \ref{prop:canonical-chi}, and \ref{prop:canonical-e}.
\end{proof}

\begin{cor}
\label{cor:rlk}
	If all orientable rulings $\rho$ of $\leg$ have nonzero resolution linking number, then $\leg$ has no orientable filling; a similar fact for $\rlk_2(\leg, \rho)$ and fillings of any type also holds.
\end{cor}

\begin{ex}
	The unique orientable ruling $\rho$ of the figure-eight knot has $\rlk(\leg, \rho) = 1$.  Thus, the figure-eight knot does not have an orientable decomposable filling.  Of course, the conclusion  follows even more easily from the fact that the maximal $tb$ of the figure-eight knot is $-3$, which cannot be $-\chi(L)$ for any surface $L$.  We will put the corollary to more subtle use in Section~\ref{sec:torus}.
\end{ex}

\section{Rigidity of the Euler Number}
\label{sec:rigid-e}


If $L$ is a filling of $\leg$, the relation $tb(\leg) = -\chi(L) - e(L)$ from Proposition~\ref{prop:tb-euler} leaves open the possibility of non-orientable fillings realizing infinitely many distinct normal Euler numbers. On one hand, infinitely many normal Euler numbers may, indeed, be realized by smooth fillings of any knot. By connected summing any smooth slice surface of a knot with a smooth nonorientable surace of arbitary Euler number in the four-ball, every Euler number within the constraints of \cite{Massey:Euler} may be realized.  In the decomposable setting, however, this is not possible, as stated in Proposition~\ref{prop:finite-euler}.

\begin{proof}[Proof of Proposition~\ref{prop:finite-euler}]
	A front diagram of a Legendrian knot supports only finitely many rulings, and hence only finitely many Euler characteristics of decomposable fillings by Proposition~\ref{prop:canonical-chi}.  It follows that there are only finitely many possible normal Euler numbers.
\end{proof}

\begin{ex}
We claim that any torus knot of the form $T(-p,2)$ has only one Euler number realized by a decomposable Lagrangian filling.  Indeed, it is straightforward to check that such a knot has a unique non-orientable ruling.  Thus, every decomposable filling of $T(-p,2)$ --- and there is at least one such --- must have the same Euler characteristic and hence the same Euler number.
\end{ex}

\begin{figure} 
\centerline{\includegraphics{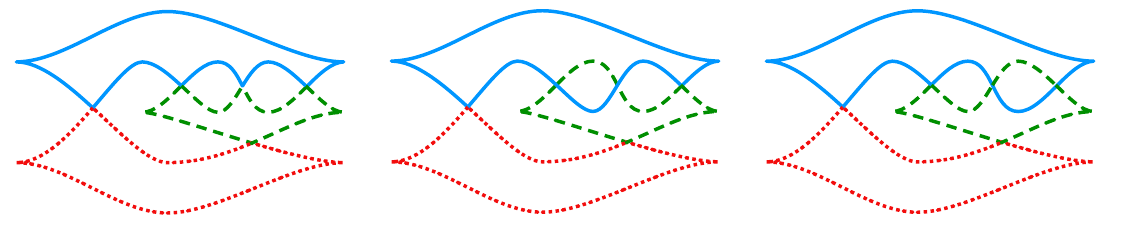}}
\caption{The three rulings of a maximal $tb$ front diagram for a mirror $5_2$ knot.}
\label{fig:5_2}
\end{figure}

\begin{ex}
We claim that the Legendrian front $\leg$ of the $m(5_2)$ knot shown in Figure~\ref{fig:5_2} has only one Euler number realized from a decomposable Lagrangian filling.  It is straightforward to compute that $\leg$ has exactly three rulings, as pictured in the figure.  We compute that the rightmost two rulings have $\rlk_2 = 1$, so by Theorem~\ref{thm:rlk}, we know that neither of these rulings can correspond to a decomposable filling.  Thus any decomposable Lagrangian filling of $\leg$ may only realize a single Euler number.

In fact, as we shall see in Section~\ref{sec:plus-ad}, since the $5_2$ knot is a non-positive alternating knot, it must have a non-orientable filling corresponding to the ruling which switches at every crossing.
\end{ex}

Although we will not go into detail here, one can also use other techniques to restrict the number of possible normal Euler numbers of a non-orientable Lagrangian filling.  For example, we may use an ungraded version of the Seidel Isomorphism as in \cite{4-plat} to restrict the possible topologies, and hence Euler numbers, of non-orientable fillings.  

In contrast to the examples above, there is no universal upper bound on the cardinality of the set of normal Euler numbers realized by exact non-orientable fillings. The underlying idea is to iterate the $tb$-twisted Whitehead double construction, denoted $Wh(\leg)$ and illustrated in Figure~\ref{fig:wh}. We observe that each iterated double has an orientable torus filling (which by necessity has Euler number 0), and then we inductively promote each fillable ruling for an iterated double to a fillable ruling on the next iterate with increased Euler number. In particular, the desired set of fillings for the $n^\text{th}$ double consists of the torus filling, together with the set of inductively promoted torus fillings of each previous double.

\begin{figure}
	\centerline{\includegraphics{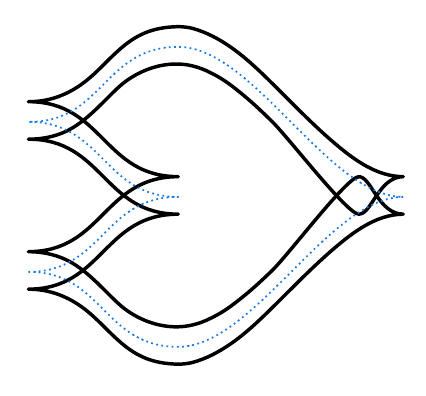}}
	\caption{The $tb$-twisted Whitehead double $Wh(\leg)$, in solid black, of a Legendrian knot $\leg$, in dotted blue.}
	\label{fig:wh}
	\end{figure}

\begin{thm} \label{thm:wh}
	The sequence of Legendrian links $\leg_0, \leg_1, \leg_2, \dots$ such that $\leg_n = Wh(\leg_{n-1})$ has the property that, for $n \geq 1$, $\leg_n$ has at least $n$ rulings induced by fillings with distinct non-negative Euler numbers.
\end{thm}
	
\begin{proof}
	To begin the construction, let $\leg_0$ be any Legendrian knot, and let $\leg_1 = Wh(\leg_0)$; we may place the clasp near a right cusp of $\leg_0$ as in Figure~\ref{fig:wh}.  As shown in \cite[\S5.1]{bst:construct}, the Legendrian $\leg_1$ has a decomposable orientable filling by a Lagrangian torus $L$; in particular, $e(L) = 0$.
	
	Working inductively, suppose that $\leg_{n-1}$ is a $tb$-twisted Whitehead double that has decomposable fillings $\{L_{1}, \ldots, L_{n-1}\}$ with distinct Euler numbers $e_i = e(L_i)$.  Denote by $\rho_i$ the canonical ruling induced by $L_i$. For each $L_i$ and $\rho_i$, we will construct a decomposable filling $\hat{L}_i$ for $\leg_n = Wh(\leg_{n-1})$ with canonical ruling $\hat{\rho}_i$ and Euler number $\hat{e}_i = e_i + 4$.
	
	We begin by constructing $\hat{\rho}_i$ on $\leg_n$.  As shown in Figure~\ref{fig:construct-wh-thread}, start with the diagram of $\leg_{n-1}$ and a maximal unknot just below and to the right of the upper cusp near the clasp.  Thread the unknot through the center of the Whitehead double pattern as in the figure, initially bypassing the clasp of $\leg_{n-1}$ and creating new ruling disks and switches when passing through a pair of cusps of $\leg_{n-1}$, until the process returns to the clasp.  Correct the diagram near the clasp as in Figure~\ref{fig:construct-wh-clasp}, defining the ruling $\hat{\rho}_i$ depending on the structure of $\rho_i$ near the clasp as in the figure.
	
	\begin{figure}
	\centerline{\includegraphics{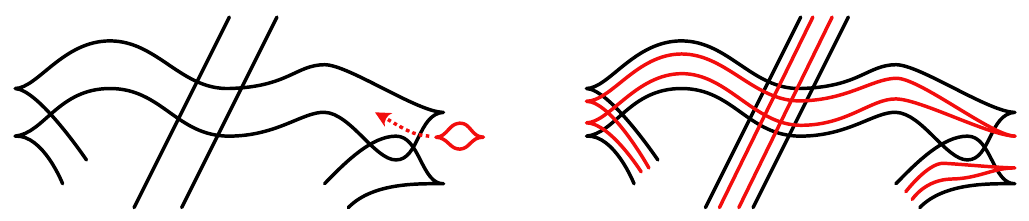}}
	\caption{Threading an unknot through a Whitehead double.}
	\label{fig:construct-wh-thread}
	\end{figure}
	
	\begin{figure}
	\centerline{\includegraphics{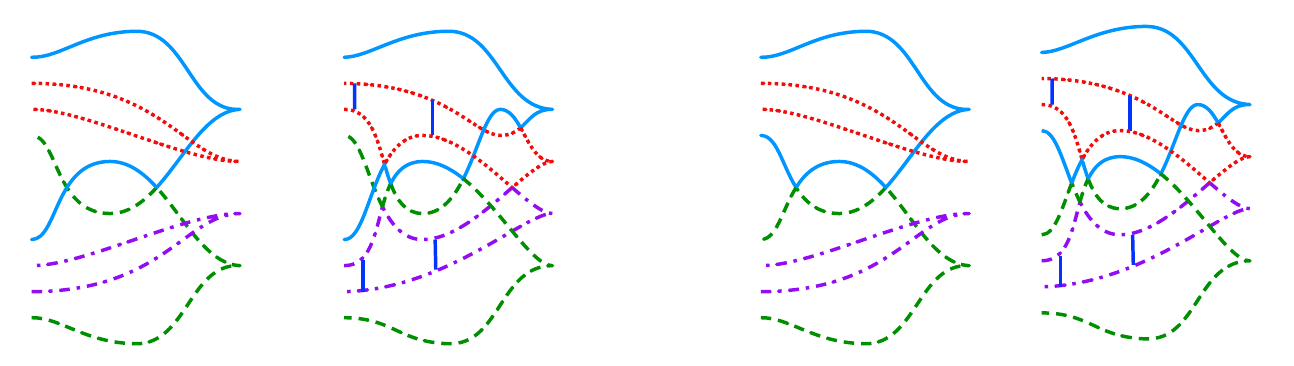}}
	\caption{Completing the threaded diagram to an iterated Whitehead double, complete with new rulings depending on the existing ruling near the old clasp.}
	\label{fig:construct-wh-clasp}
	\end{figure}
	
		It is straightforward to check that the process above yields a front diagram for $\leg_n = Wh(\leg_{n-1})$ with a ruling $\hat{\rho}_i$.  To construct $\hat{L}_i$, pinch across $\hat{\rho}_i$ just to the right of each of the four switches of the inner ruling disks (see Figure~\ref{fig:construct-wh-clasp}), then unthread the resulting unknot from $\leg_{n-1}$, and finally remove the unknot with a $0$-handle move.  We have returned to $\leg_{n-1}$ with $\rho_i$, which we know is the canonical ruling for $L_i$; we complete $\hat{L}_i$ by concatenating with $L_i$.  Lastly, we may compute from counting extra switches near the clasps in Figures~\ref{fig:construct-wh-thread} and \ref{fig:construct-wh-clasp} that $\chi(\hat{\rho}_i) = \chi(\rho_i) - 4$.  It follows from Proposition~\ref{prop:tb-euler} that $e(\hat{L}_i) = e(L_i) +4$.  
		
		In all, the collection of fillings $\{\hat{L}_1, \ldots, \hat{L}_{n-1}, L_n\}$, where $L_n$ is the standard orientable genus $1$ filling of $\leg_n$, yields $n$ different Euler numbers $\{e_1+4, \ldots, e_{n-1}+4, 0\}$.
\end{proof}

\begin{rem}
	When $\leg_0$ is the maximal $tb$ unknot, we see that $\leg_2$ (the Whitehead double of the trefoil) has both oriented and unoriented fillings.  This construction (not coincidentally, we believe) coincides with Sivek's construction of examples of Legendrians whose Chekanov-Eliashberg DGAs have  augmentations that induce multiple distinct linearized Legendrian contact homologies \cite{sivek:bordered-dga}.  Further, the first example of such a Legendrian, due to Melvin and Shrestha \cite{melvin-shrestha}, is the $m(8_{21})$ knot in Example~\ref{ex:8-21-filling}, which also has both orientable and  non-orientable fillings.
\end{rem}

\section{Plus-Adequate Knots}
\label{sec:plus-ad}


We shift our attention from obstructions for the existence of non-orientable fillings to constructions of such fillings.  The underlying goal is to uncover hints of the geography of non-orientably fillable smooth knots, which appears to be more complicated than that of orientably fillable knots.  In this section, we concentrate on plus-adequate knots, which simultaneously generalize positive and alternating knots; we will obtain proofs of Theorems~\ref{thm:positive} and \ref{thm:alternating} as corollaries of the main work in this section.

We begin by defining plus-adequate knots as in \cite{lt:plus-adequate}.  Given a smooth knot diagram $D$, we form $s_+(D)$ by resolving all of the crossings as in Figure~\ref{fig:ruling-defn}(a).  The diagram $D$ is \dfn{plus-adequate} if the two segments in the resolution near every crossing belong to different components of $s_+(D)$; a knot $K$ is plus-adequate if it has a plus-adequate diagram.

\begin{thm} \label{thm:plus-ad}
	Every plus-adequate knot $K$ has a Legendrian representative $\leg$ with a decomposable filling. The filling is orientable if and only if $K$ is positive. 
\end{thm}

\begin{proof}
	K\'alm\'an showed that every plus-adequate knot $K$ has a Legendrian representative $\leg$ with a normal ruling $\rho$ in which all crossings of $\leg$ are switches of $\rho$ \cite{kalman:+adequate}. Carrying out the procedure in \cite[\S5]{positivity} without regard for orientation, we obtain a decomposable filling $L$ of $\leg$ with canonical ruling $\rho$.  

	If $L$ is orientable, then by Proposition~\ref{prop:ori-canonical-ruling}, we must have that $\rho$ is orientable.  Since $\rho$ has a switch at every crossing, we see that the front diagram for $\leg$ has only positive crossings, and hence $K$ is positive.  Conversely, if $K$ is positive, then as noted in \cite{tanaka:max-tb-pos}, there exists a Legendrian representative $\leg$ of $K$ that has only positive crossings.  Thus, every ruling of $\leg$ is orientable, so Proposition~\ref{prop:ori-canonical-ruling} implies that there are no non-orientable decomposable fillings of $\leg$.  Using Corollary~\ref{cor:ruling-poly-obstr}, and Remark~\ref{rem:ruling-homfly-k} in particular, we see that the result holds for \emph{any} Legendrian representative of $K$.
\end{proof}

Theorems~\ref{thm:positive} and \ref{thm:alternating} follow as immediate corollaries. Alternating knots also provide evidence for the conjecture about minimal crosscap numbers mentioned in the introduction. 

\begin{prop} \label{prop:alt-min}
	Let $L$ be a non-orientable decomposable filling with Euler number $e$ of an alternating Legendrian knot $\leg$ whose canonical ruling is switched at every crossing in a reduced alternating diagram.  The filling $L$ realizes the minimal crosscap number $\gamma_4^{e}(\leg)$ of all smooth fillings of $\leg$ with normal Euler number $e$.
\end{prop}

\begin{proof}
We certainly have $\gamma_4^{e}(\leg) \leq b_1(L)$ by definition.  To obtain the reverse inequality, we use the Gordon and Litherland signature bound \cite{gl:signature}:
\begin{equation} \label{eq:gl}
\left| \sigma(\leg) + \frac{e}{2} \right| \leq \gamma_4^e(\leg).
\end{equation}
In particular, we compute the left side in terms of quantities related to $L$.  Since the canonical ruling $\rho_L$ is switched at every crossing, the number of (positive or negative) crossings coincides with the number of (positive or negative) switches.  First, by Proposition~\ref{prop:canonical-chi}, we have 
\begin{equation} \label{eq:chi-L}
\chi(L) = c(\leg) - s(\rho).
\end{equation}  
Next, by Proposition~\ref{prop:tb-euler}, we see that
\begin{equation} \label{eq:e-L}
e(L) = -\chi(L) - tb(\leg) = 2s_-(\rho).
\end{equation}
Finally, building on \cite[p.\ 1646]{lenny:khovanov}, we derive that 
\begin{equation} \label{eq:signature}
\sigma(\leg) = c(\leg) - s_+(\rho) -1.
\end{equation}  
Note that our signature is normalized opposite to that of \cite{lenny:khovanov}, with the right-handed trefoil having signature $-2$.

Inserting Equations~\eqref{eq:e-L} and \eqref{eq:signature} into the inequality \eqref{eq:gl}, and then using Equation~\eqref{eq:chi-L}, we obtain
\[b_1(L) = \left| \chi(L) - 1\right| \leq \gamma_4^{e}(\leg),\]
thus proving the proposition.
\end{proof}

\begin{rem} 
The signature bound is not always sufficient to show that non-orientable Lagrangian fillings of plus-adequate knots realize the minimal crosscap number. The Kinoshita-Terasaka knot $11n_{42}$, shown in Figure~\ref{fig:kt}, is such an example, as its canonical filling has $b_1(L) = 7$ while the signature bound only yields $\gamma_4^e(\leg) \geq 5$.
\end{rem}

\begin{figure}[t]
\includegraphics{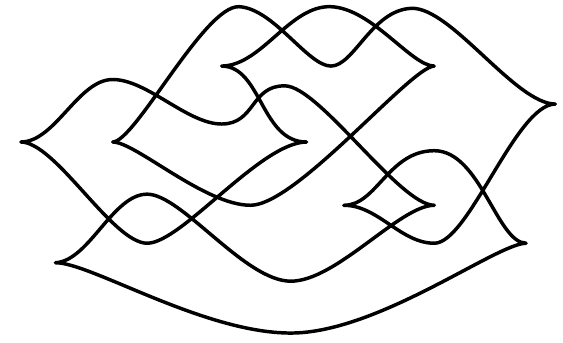}
\caption{A Legendrian Kinoshita-Terasaka $11n_{42}$ knot.}
\label{fig:kt}
\end{figure}

\section{Torus Knots}
\label{sec:torus}


This section is devoted to the proof of Theorem~\ref{thm:torus}, which characterizes the fillability of many torus knots.  Recall that we adopt the convention that for $p,q$ relatively prime with $|p|>q>0$, $T(p,q)$ denotes the $(p,q)$ torus knot.  

We may dispose of the first three claims of the theorem quickly.  If $p>q>0$, then $T(p,q)$ is positive and hence orientably fillable \cite{positivity}.  Conversely, suppose $p<0$ and $T(p,q)$ were fillable.  The maximal Thurston-Bennequin number of $T(p,q)$ is $pq < -1$ \cite{etnyre-honda:knots}, but a Legendrian representative $\leg$ with orientable filling $L$ would have $-1 \leq -\chi(L) = tb(\leg)$, a contradiction.

Second, note that for $p<0$, $T(p,2)$ is alternating and non-positive, and hence has a non-orientable filling by Theorem~\ref{thm:plus-ad}.

Third, Epstein and Fuchs \cite{epstein-fuchs} show that the Kauffman bound is not sharp for $p<0$ and $q$ odd.  Theorem~\ref{thm:ruling-poly} then shows that no Legendrian representative of $T(p,q)$ has a ruling, and hence, by Proposition~\ref{prop:ruling-obstr}, no Legendrian representative of $T(p,q)$ has a filling.

We arrive at the case where $p<0$ and $q$ is even and greater than $2$. Note that Capovilla-Searle and Traynor \cite[Corollary 1.3]{cst:non-ori-cobordism} claimed that these knots were fillable, but the claimed filling in Figure 12 only works when $p$ and $q$ are both even. Atiponrat proved that  $T(p,4)$ is not decomposably fillable in \cite{atiponrat:torus, atiponrat:obstruction}.  We will simplify Atiponrat's proof by using the unoriented resolution linking number in place of the clasp parity, and we will extend the result to encompass \emph{any} $q$ divisible by $4$.  

The proof proceeds in two steps:  first, in Lemma~\ref{lem:torus-rulings}, we show that any maximal Legendrian representative of $T(p,q)$ with $p<0$ and $q$ even has a unique ruling, which is necessarily non-orientable since $T(p,q)$ is a negative knot.  Second, in Lemma~\ref{lem:torus-rlk}, we compute the unoriented resolution linking number of this ruling, and show that it is non-zero when $4|q$.  These two lemmas, combined with Corollary~\ref{cor:rlk}, then complete the proof of the last part of Theorem~\ref{thm:torus}.

\begin{lem}
\label{lem:torus-rulings}
Any maximal Legendrian representative of $T(p,q)$ with $p<0$ and $q$ even has a unique ruling $\rho$.
\end{lem}

\begin{proof}
	By Theorem~\ref{thm:ruling-poly}, it suffices to show that the coefficient $a^{-pq-1}$ of the Kauffman polynomial of $T(p,q)$ is $z$.  By a theorem of Yokota \cite{yokota:kauffman} as interpreted by \cite{epstein-fuchs},\footnote{In \cite{epstein-fuchs}, the convention is that $p \geq q >0$ and negative torus knots are of the form $T(-p,q)$; we have changed the signs of $p$ in the formulae to conform to our convention.} the Kauffman polynomial of $T(p,q)$ when $q$ is even may be expressed as
	\[F(a,z) = \frac{a^{-pq}z}{a+z-a^{-1}}(G_{p,q}(a,z) + 1),\]
	where $G_{p,q}(a,z)$ has maximal degree $p+q$ in $a$.  Thus, the highest degree in $a^{-pq}z (G_{p,q}(a,z)+1)$ is $-pq$ and the leading term is $za^{-pq}$.  It follows from the formula above that the leading term of $F(a,z)$ is $za^{-pq-1}$, as required.
\end{proof}

We proceed to construct the ruling $\rho$ promised by Lemma~\ref{lem:torus-rulings}.  Write $|p| = (n_1+n_2+1)q + r$ for some non-negative $n_i$.  Any maximal Legendrian representative of $T(p,q)$ with $p<0$ is isotopic to a Legendrian with a front diagram as in Figure~\ref{fig:torus-front} for some choice of $n_1$ and $n_2$ \cite{etnyre-honda:knots}. 

\begin{figure}
\labellist
\small\hair 2pt
 \pinlabel {$q$} [t] at 115 144
 \pinlabel {$n_2q$} [l] at 69 95
 \pinlabel {$(n_1+1)q+r$} [l] at 231 95
\endlabellist
	\centerline{\includegraphics{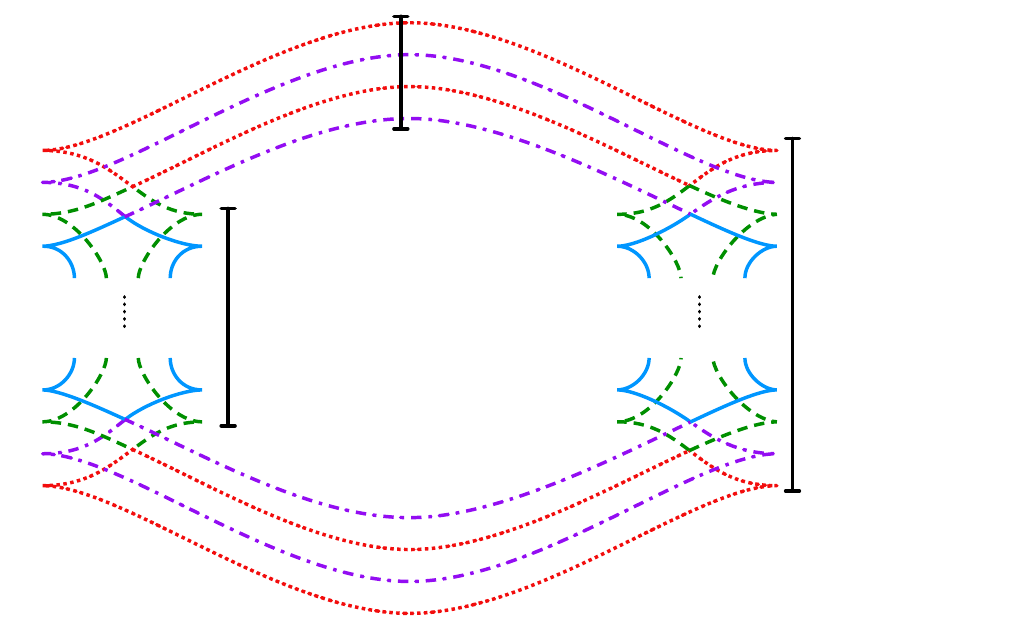}}
	\caption{A front diagram for a Legendrian negative torus knot $T(p,q)$ with $|p| = (n_1+n_2+1)q + r$.  The unique ruling of this front is indicated.}
	\label{fig:torus-front}
\end{figure}

To produce the ruling $\rho$, first consider the strands at the top of the diagram in Figure~\ref{fig:torus-front}, numbering them from top to bottom.  For $k \leq \frac{q}{2}$, pair strands $k$ and $\frac{q}{2}+k$, adding a switch where they meet in the left- and right-hand lattices of crossings.  Add switches at all crossings in the same columns as the switches arising from the top strands. This produces ruling disks shaped like rhombi in the crossing lattices and ruling disks across the bottom strands in the same pattern as at the top.  To check that this is a normal ruling, notice that each pair of ruling disks incident to every switch is disjoint.

\begin{lem}
\label{lem:torus-rlk}
	Let $\leg$ be maximal Legendrian representative of $T(p,q)$ with $p<0$ and $q$ even, and let $\rho$ be the unique ruling of $\leg$.  We may compute that
	\[\rlk_2(\rho) \equiv \begin{cases} 0 & q \equiv 2  \mod 4, \\ 1 & q \equiv 0  \mod 4. \end{cases} \]
\end{lem}

\begin{proof}
	The front diagram of $\leg$ in Figure~\ref{fig:torus-front} has $|p|(q-1)$ crossings and $|p|$ right cusps.  It follows that there are $|p|$ ruling disks and $|p|$ switches for $\rho$, as each ruling disk has two switches and each switch is incident to two ruling disks.  Thus, the link $\leg^\rho$ has $|p|(q-2)$ crossings.  Orienting all of the components of $\leg^\rho$ counterclockwise, we see that each two-component sublink of $\leg^\rho$ is either the unlink or a negative Hopf link; in particular, all of the crossings of $\leg^\rho$ are negative.  Thus, we see that 
	\[\rlk_2 (\rho) \equiv p \left(\frac{q}{2}-1\right) \mod 2,\]
	and the lemma follows.
\end{proof}

We end this section with a conjecture about the remaining cases.

\begin{conj}
	The only non-orientably fillable torus knots are of the form $T(p,2)$ with $p<0$.
\end{conj}

To prove such a conjecture, one would need a more powerful invariant than $\rlk_2$.  An integral lift of $\rlk_2$ using a formula similar to that in Lemma~\ref{lem:rlk-alt}, which would include a combinatorial translation of the normal Euler number for a ruling, may suffice.

\section{Pretzel Knots}
\label{sec:pretzel}


This section is devoted to the proof of Theorem~\ref{thm:pretzel}, which determines which $3$-stranded pretzel knots are non-orientably fillable, except for one small family. Let $P(\pm p_1,\pm p_2,\pm p_3)$ for $p_1,p_2,p_3>0$ be the 3-stranded pretzel knot, where each twist box contains $p_i$ half twists with  a positive (resp. negative) sign indicating left-handed (resp. right-handed) twists. We will make frequent use of the well-known fact that if $P(-p_1,p_2,p_3)$ is a knot, then at most one of $p_1,p_2,p_3$ is even.

Front diagrams for all maximal $tb$ Legendrian pretzel knots were determined by Ng in \cite[\S6]{lenny:thesis}, except for the family $P(-p_1,-p_2,p_3)$ with $p_1\geq p_2=p_3+1$. The maximal $tb$ of this exceptional case is still unknown.

Any 3-stranded pretzel knot may be presented as one of the five forms listed below, with maximal $tb$ front diagrams for each of the five forms shown in Figures~\ref{pretzel_fronts123} and \ref{pretzel_fronts45}.
\begin{enumerate}
\item $P(p_1,p_2,p_3)$
\item $P(-p_1,-p_2,-p_3)$
\item $P(-p_1,p_2,p_3)$
\item $P(-p_1,-p_2,p_3)$ with $p_1\geq p_2$, $p_2\leq p_3+1$
\item $P(-p_1,-p_2,p_3)$ with $p_1\geq p_2\geq 2$ and $p_2\geq p_3+2$
\end{enumerate}

These five forms correspond exactly to the cases in Theorem \ref{thm:pretzel}. Note that the exceptional family is in case (4) with $p_2=p_3+1$, and this family is exactly the one mentioned in the statement of Theorem \ref{thm:pretzel}.

\begin{rem}In cases (1)--(2), in case (3) with $det(K)<0$, and in cases (4)--(5) with $det(K)>0$, the $4$-dimensional crosscap number is minimized by the Lagrangian filling among surfaces with the same Euler number. One can check this by appealing to Proposition~\ref{prop:alt-min} for the first two cases, and calculating the signature of each knot in the last three cases using the Goeritz matrix.
\end{rem}

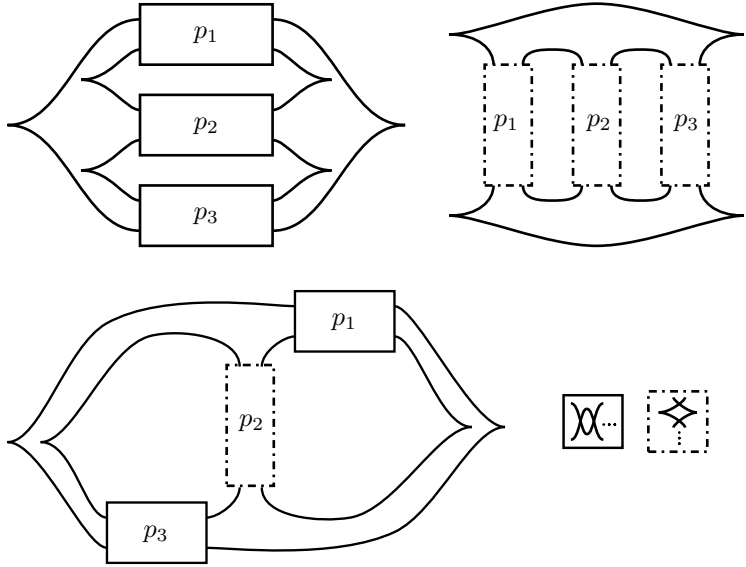
\begin{figure}
\labellist
\small\hair 2pt
 \pinlabel {$p_1$}  at 80 204
 \pinlabel {$p_2$}  at 80 170
 \pinlabel {$p_3$} at 80 136
 \pinlabel {$p_1$} at 193 171
 \pinlabel {$p_2$} at 228 171
 \pinlabel {$p_3$} at 261 171
 \pinlabel {$p_1$} at 132 96
 \pinlabel {$p_2$} at 97 58
 \pinlabel {$p_3$} at 61 16
\endlabellist 
  \centerline{\includegraphics{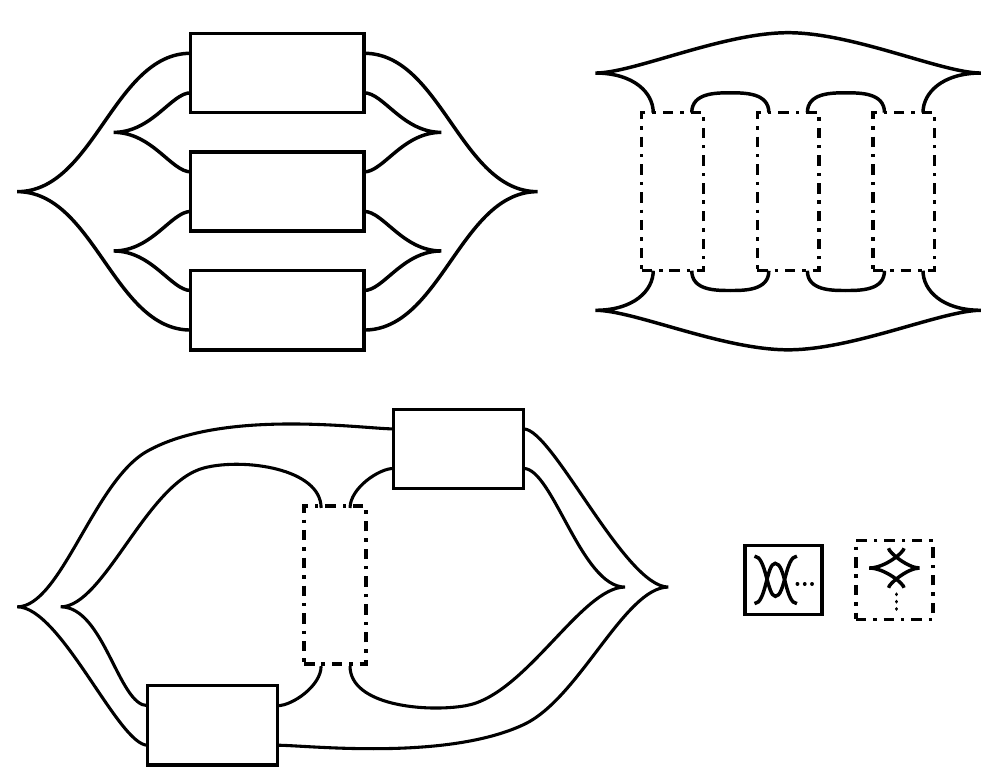}}
  \caption{Maximal $tb$ front diagrams for $P(p_1,p_2,p_3)$ (top left), $P(-p_1,-p_2,-p_3)$ (top right), and $P(-p_1,p_2,p_3)$ (bottom).}
  \label{pretzel_fronts123}
\end{figure}

\begin{figure}
\labellist
\small\hair 2pt
 \pinlabel {$a$} at 43 141
 \pinlabel {$b$} at 105 141
 \pinlabel {$c$} at 150 129
 \pinlabel {$p_3$}  at 207 143
 \pinlabel {$p_1$} at 315 132
 \pinlabel {$p_3$}  at 72 73
 \pinlabel {$d$}  at 105 43
 \pinlabel {$e$} at 135 43
\endlabellist
  \centerline{\includegraphics{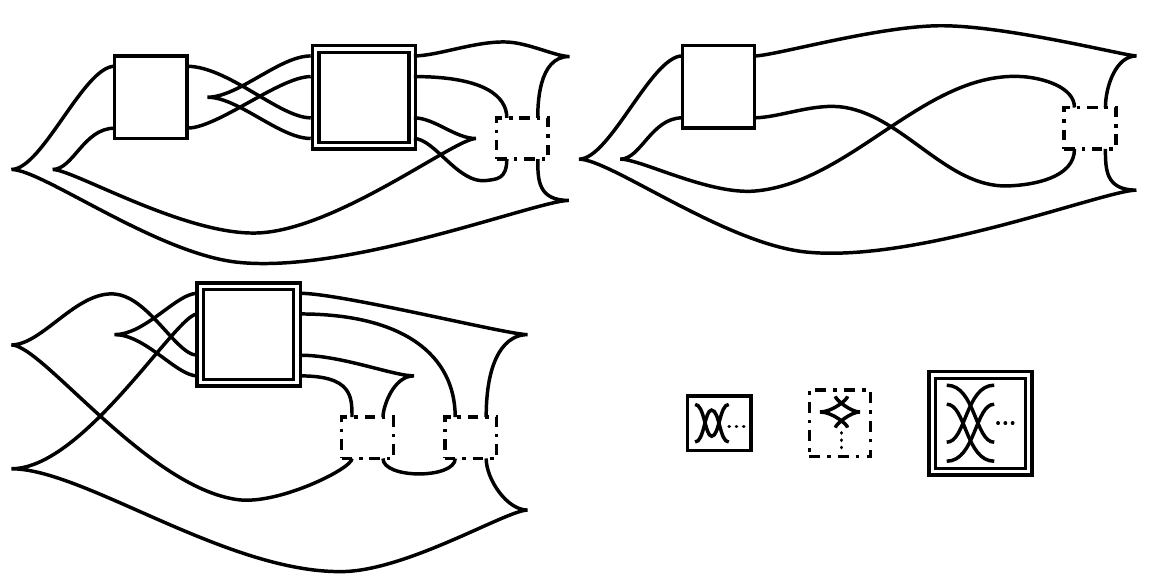}}
  \caption{Front diagrams for $P(-p_1,-p_2,p_3)$ with $p_1\geq p_2\geq 2$ and either $p_2  < p_3+1$ (top left with $a=p_3-p_2+1$, $b=p_2-2$, $c=p_1-p_2+1$), $p_2 = p_3+1$  (top right); or $p_2\geq p_3+2$ (bottom left with $d=p_2-p_3-1$, $e=p_1-p_3-1$). These have maximal $tb$, with the possible exception of the $p_2 = p_3+1$ case.}
  \label{pretzel_fronts45}
\end{figure}
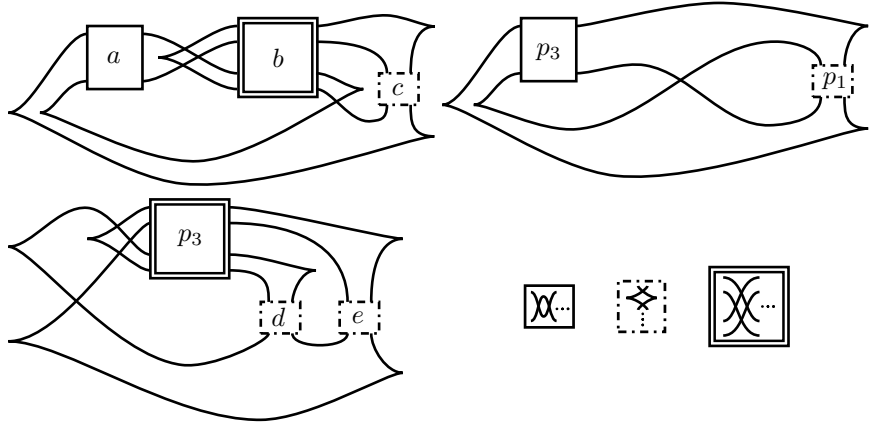

\begin{proof}[Proof of Theorem~\ref{thm:pretzel}]

\textbf{Case (1): $P(p_1,p_2,p_3)$.} Since these pretzel knots are all non-positive and alternating, Theorem \ref{thm:alternating} implies that each knot in the family admits a decomposable non-orientable filling.
 
\textbf{Case (2): $P(-p_1,-p_2,-p_3)$.}  These pretzel knots are all alternating and are positive if and only if all of the $p_i$ are odd. It follows that if one of the $p_i$ is even, then $P(-p_1,-p_2,-p_3)$ admits a decomposable non-orientable filling. Further, if all the $p_i$ are odd then $P(-p_1,-p_2,-p_3)$ admits no decomposable non-orientable filling by Theorem \ref{thm:positive}.

\textbf{Case (3): $P(-p_1,p_2,p_3)$.} Further, one can easily check that if $p_1$ is even then $P(-p_1,p_2,p_3)$ is positive, in which case it has no decomposable non-orientable filling. Conversely, if $p_1$ is odd then $P(-p_1,p_2,p_3)$ is non-positive and plus-adequate by work of \cite{FKP:guts}, and therefore admits a decomposable non-orientable filling by Theorem \ref{thm:plus-ad}.

For the last two cases (4) and (5), we will explicitly construct fillings for the knots in the families of Theorem~\ref{thm:pretzel}, and show that all other knots in each case do not have non-orientable normal rulings. Theorem~\ref{prop:ori-canonical-ruling} will then imply the desired result. Before considering each case individually, note first that in either case (4) or (5), the cusps inside the vertical (dashed) twist boxes at the right of the diagram must form ruling disks in pairs. In both cases, there are exactly three left and right cusps not contained in twist boxes, so we will consider how these cusps may form ruling disks.

We will start by examining the vertical slice of the front diagram shown in Figure \ref{NNOF1}, given by the double-stranded twist box and the two parallel strands below it. The idea of the proof in both cases (4) and (5) is to restrict the rulings on this region which could possibly extend to normal rulings on the entire knot. Then, we will build normal rulings for the pretzel knots in cases (4) and (5) from the rulings in this region, and examine the (non)-orientability of these rulings to get restrictions on the $p_i$. Finally, for every case which admits a non-orientable normal ruling, we will explicitly construct a decomposable non-orientable filling associated to that ruling.

\begin{figure}
  \centerline{\includegraphics{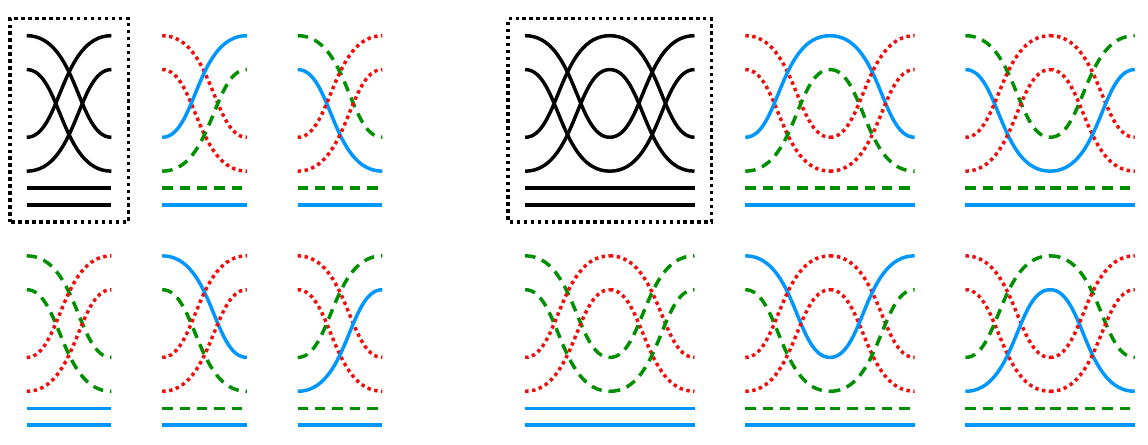}}
  \caption{A vertical slice of a front diagram for Case (4) and (5) pretzel knots with possible normal rulings. A single twist is shown on the left, and two twists on the right.}
  \label{NNOF1}
\end{figure}

In Figure \ref{NNOF1}, there are six strands in total, and each top and bottom strand of each of the three ruling disks not contained in twist boxes must pass through the region depicted. In particular, each strand in the region corresponds to a strand of one of these ruling disks. With a single double-stranded twist, one may check that there are only five ways of pairing these strands to form (potentially) normal rulings, shown on the left of Figure \ref{NNOF1}. Though we omit, for simplicity, the cases which do not yield normal rulings, one could verify this claim by first observing that there can be no switches in this region, and second, checking that (up to relabeling), the cases shown are the only non-switched normal rulings for this region. It follows that for arbitrary numbers of double-stranded twists, there are also only five potential rulings, given by juxtaposing possible rulings on each individual twist, shown on the right of Figure \ref{NNOF1}.

\textbf{Case (4): $P(-p_1,-p_2,p_3)$ with $p_1\geq p_2$ and $p_2\leq p_3+1$.}

\begin{figure}
\centerline{\includegraphics{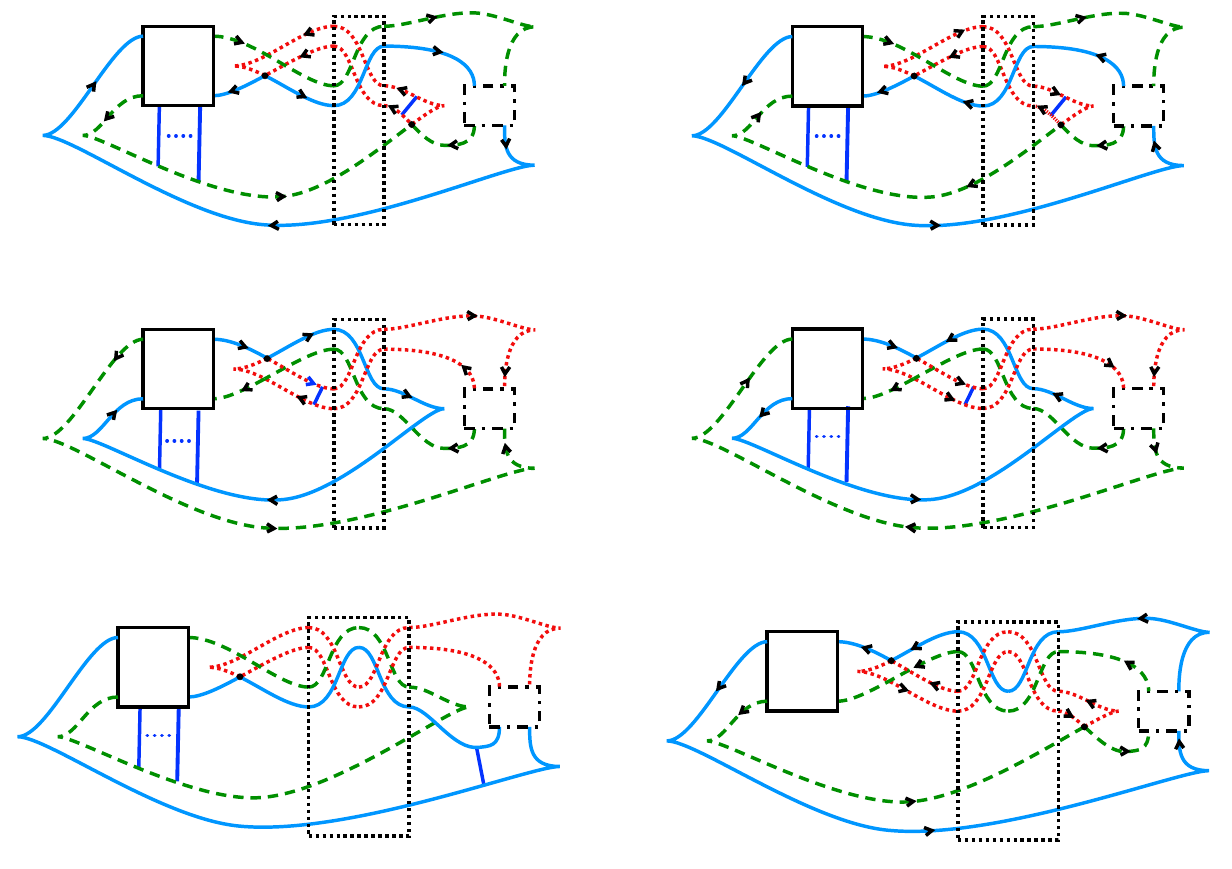}}
\caption{Normal rulings for case (4) pretzel knots with $p_2 \geq 2$.}
\label{case4}
\end{figure}

First, for the case $p_2>1$, the possible normal rulings built from the regions in Figure \ref{NNOF1} are shown in Figure \ref{case4}, where $p_2$ is assumed to be odd in the top two rows and even in the bottom row. Note that only the cases $p_2=3,4$ are shown, for simplicity, though the same arguments apply with arbitary $p_2$ of the same parity. One can check that, in the vertical slice from Figure \ref{NNOF1} (indicated by the dashed box in Figure \ref{case4}), the rulings which are not included cannot be extended to normal rulings in this case. Also, in the top two rows, the rulings in each row are the same up to change of orientation outside the twist boxes, though the given orientations are determined by the $p_i$. We will now examine what conditions on the $p_i$ are necessary to realize the rulings shown, with the given orientations, and when these rulings may be non-orientable.

For the top left ruling in Figure \ref{case4} to have the given orientation, we must have that $p_2$ is odd, $p_3-p_2+1$ is even and at least two (so that the ruling can switch at one of the crossings in the left twist box), and that $p_1=p_2$ (so that there is only a single crossing in the right twist box). In particular, we have that $p_1$ and $p_2$ are odd, $p_3$ is even, $p_2<p_3$ and $p_1=p_2$. For the top right, the conditions are similar except that $p_3$ must be odd instead. Here, the ruling is non-orientable only if it switches at (at least) one of the crossings in the left twist box. Whenever these conditions are met, the pinch moves shown yield non-orientable fillings.

For the left ruling in the second row in Figure \ref{case4} to have the given orientation, we must have that all of the $p_i$ are odd, and for it to be non-orientable we must have that $p_3-p_2+1\geq 3$ (so that the ruling can switch at one of the crossings in the left twist box). In particular, we have that all the $p_i$ are odd and $p_2<p_3$. For the right ruling in that row, the conditions are the same except that $p_3$ must be even instead. Again, the pinch moves shown yield non-orientable fillings.

For the left ruling in the third row in Figure \ref{case4} to be non-orientable, we must have that $p_2$ is even, $p_1$ and $p_3$ are odd, and $p_3-p_2+1\geq 1$. Because $p_2$ and $p_3$ have opposite parity, this implies $p_2<p_3$. So, we must have that $p_2$ is even, $p_1$ and $p_3$ are odd, and $p_2<p_3$. The ruling on the right is always orientable, and hence is not associated to a non-orientable filling. As above, whenever these conditions are met, the pinch moves shown yield non-orientable fillings.

\begin{figure}
\centerline{\includegraphics{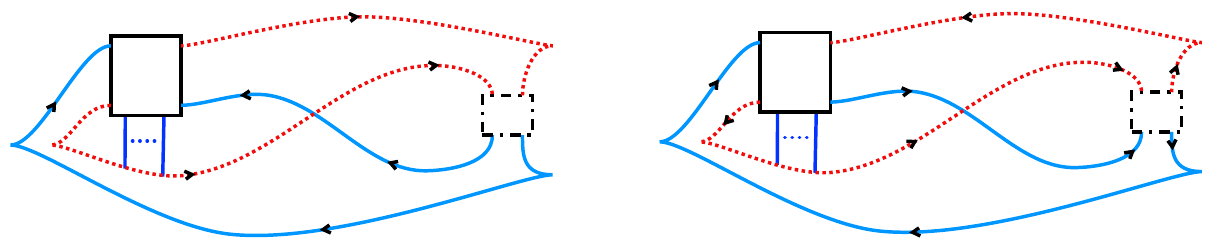}}
\caption{Normal rulings for case (4) pretzel knots with $p_2 =1$.}
\label{case4-1}
\end{figure}

Finally, we will deal with the case where $p_2=1$. The only normal ruling in this case (outside of the twist boxes) is the one shown in Figure \ref{case4-1}. One may check that the only two possible orientations on the strands of the knot which make this ruling non-orientable are the ones shown in the Figure. For the orientation on the left, we must have that $p_3$ is even and this implies that $p_1$ must be odd for $K$ to be a knot. For the orientation on the right, we must have that $p_1$ and $p_3$ are both odd. In particular, the knot has a non-orientable normal ruling exactly when $p_1$ is odd, regardless of the parity of $p_3$. For these rulings, the pinch moves shown yield non-orientable fillings.

Finally, we may verify that in each case, we must have $p_2<p_3$ and $p_1$ is odd, except if $p_2=p_3=1$ in the $p_2=1$ case. Furthermore, any tuple $(p_1,p_2,p_3)$ satisfying at most one of $p_2$ or $p_3$ is even (so that the corresponding pretzel knot is a knot), $p_1$ is odd, $p_1\geq p_2\geq 1$, and $p_2<p_3$ (or $p_2=p_3=1$) satisfies the conditions necessary to realize one of the non-orientable rulings from Figure \ref{case4}.

\begin{rem}
Note that, whenever $K=P(-p_1,-p_2,p_2-1)$ has a maximal $tb$ front of the form shown in Figure \ref{pretzel_fronts45}, the proof above shows that $K$ does not have a decomposable, non-orientable filling. If we had a maximal $tb$ front for this exceptional case, the techniques of this section could likely be used to classify when this final case is fillable, thereby completing the fillability classification of all 3-stranded pretzel knots.
\end{rem}

\textbf{Case (5): $P(-p_1,-p_2,p_3)$ with $p_1\geq p_2\geq 2$ and $p_2\geq p_3+2$.}

\begin{figure}
\centerline{\includegraphics[width=\textwidth]{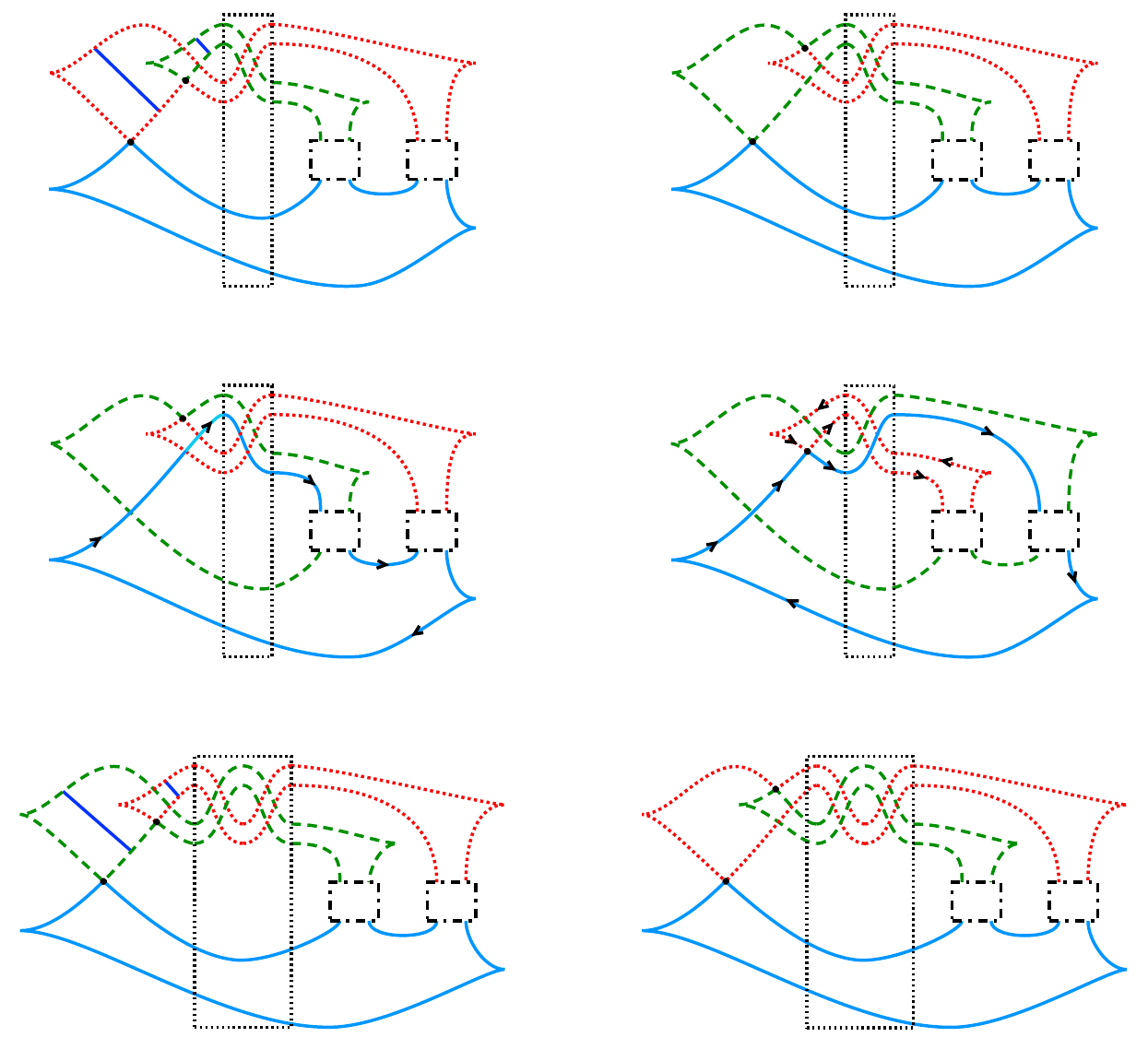}}
\caption{Normal rulings for case (5) pretzel knots.}
\label{case5}
\end{figure}

The possible normal rulings built from the regions in Figure \ref{NNOF1} are shown in Figure \ref{case5}, where $p_3$ is odd in the top two rows and even in the bottom row. One may easily verify that the regions from Figure \ref{NNOF1} not included in Figure \ref{case5} cannot be extended to normal rulings. The two rulings in the middle row in Figure \ref{case5} are orientable (regardless of the orientation on the green strands), and the rulings on the right in each row are all related to their adjacent rulings on the left by R0 moves.

In the top left ruling in Figure \ref{case5}, we must have that $p_3$ is odd, and one can check that this ruling is nonorientable exactly when one of $p_1$ or $p_2$ is even. When these conditions are met, the pinch moves shown yield nonorientable fillings.

In the bottom left ruling in Figure \ref{case5}, we must have that $p_3$ is even, and for the pretzel knot to be a knot it follows that $p_1$ and $p_2$ must both be odd. In this case, one can check that the ruling is nonorientable, and the pinch moves shown yield nonorientable fillings.

Finally, as before, any tuple $(p_1,p_2,p_3)$ with exactly one of the $p_i$ even, $p_1\geq p_2\geq 2$, and $p_2\geq p_3+2$ satisfies the conditions necessary to realize one of the non-orientable rulings in Figure \ref{case5}.
\end{proof}

\bibliographystyle{amsplain} 
\bibliography{main}

\end{document}